\documentclass[11pt]{article}

\usepackage[margin=1in]{geometry}
\usepackage{amsmath,amssymb,amsthm}
\usepackage[square]{natbib}
\usepackage{mathtools}
\usepackage{stmaryrd}
\usepackage{hyperref}
\usepackage{enumitem}
\usepackage{microtype}

\hypersetup{colorlinks=true,linkcolor=blue,citecolor=blue,urlcolor=blue}
\usepackage{xcolor}
\newtheorem{definition}{Definition}[section]
\newtheorem{theorem}[definition]{Theorem}
\newtheorem{proposition}[definition]{Proposition}
\newtheorem{lemma}[definition]{Lemma}
\newtheorem{corollary}[definition]{Corollary}
\theoremstyle{remark}

\title{Models of Arbitrary Sets}
\author{Pedro Teixeira Yago\thanks{Orcid: 0000-0001-7993-4516, Classe di Lettere e Filosofia, Scuola Normale Superiore di Pisa, Italy. \texttt{pedro.tyago@outlook.com}}
\and Santiago Jockwich Martinez\thanks{Orcid: 0000-0001-6528-1972, Department of Philosophy, University of Gdansk, Poland. \texttt{santijoxi@hotmail.com}}
\and Giorgio Venturi\thanks{Orcid: 0000-0001-9089-4340, Dipartimento di Civilit\`a e Forme de Sapere, Universit\`a di Pisa, Italy. \texttt{gio.venturi@gmail.com}}}
\date{}

\begin{document}
\maketitle

\begin{abstract}
There is a growing literature on arbitrary objects as other types of mathematical objects, besides particular mathematical objects. In this paper we study arbitrary sets in the context of ZFC set theory. An arbitrary object is one that is determined by a state space across which it varies and by its instantiations. 
We compare two semantic strategies. First, we analyse the possibility of modeling arbitrary sets using direct powers of a ZFC structure. 
This yields a natural  setting, but one where  Extensionality and, crucially, the full Separation scheme fail. Second, we consider structures composed of arbitrary sets using a direct power construction evaluated in a Boolean-valued semantics. Within this setting, we develop the cumulative hierarchy of arbitrary sets, arbitrary ordinals, cardinals, and filtrations of the structure of arbitrary sets. The final section develops a combinatorial interpretation in terms of independent assignments across a fixed state space. Its connection with the iterative conception is then discussed. 
\end{abstract}

\medskip
\noindent\textbf{Keywords.} Arbitrary objects; Arbitrary sets; Boolean-valued semantics; Combinatorial conception; ZFC.\\

\section{Introduction}

Arbitrary objects, in general, are entities that assume values at a state space. The values an arbitrary object may assume compose what is usually called its \emph{scope} (or range) and determine many of its characteristic properties. A convenient mathematical representation is as follows: given a state space $\Omega$ and a scope $F$, an arbitrary $F$-object can be represented by a function $f : \Omega \to F$, with $f(i)$ the value of the object at state $i$.

A central idea about arbitrary objects, clearly expressed in \citep{Horsten2019}, is that the state space must be large enough to support \emph{general} arbitrary objects. In particular, if $F$ is the intended scope, then the state space must allow at least one arbitrary object whose scope is \emph{all of $F$}. This becomes a cardinality requirement: to have an object whose scope is the whole of $F$, one needs a state space $\Omega$ with $|\Omega| \geq |F|$ (in the metatheory).

In a recent paper, Horsten argues that Boolean-valued sets can be understood as arbitrary objects \citep{Horsten2024}. The key observation is that a Boolean-valued set $u$ determines (and is determined by) an associated function $u^\ast$ whose domain is the underlying Boolean algebra, and whose codomain is the Boolean-valued model. Thus Boolean-valued models provide a natural \emph{function-theoretic} representation of arbitrary objects. However, Boolean-valued sets cannot, in general, represent any arbitrary set of a given Boolean-valued model, since the underlying Boolean algebra need not be large enough to support a \emph{general} arbitrary set whose scope is the entire set-theoretic totality under consideration.

The aim of this paper is to explore function-based constructions, representing arbitrary sets, that \emph{do} accommodate the general-arbitrary-set requirement and that make the resulting mathematics explicit. Furthermore, we develop a combinatorial interpretation of arbitrary sets through the operator $X\mapsto X^\Omega$, where $X$ is a set of admissible values and $\Omega$ is a fixed state space. The ground hierarchy provides a stratification of the uniformly bounded arbitrary sets, while the full domain, including functions unbounded in that hierarchy, is obtained by an external function-space construction.

Our guiding choice is to work relative to a set model $\mathcal{M}$ and to take the state space to be indexed (or represented) by $M$, the domain of $\mathcal{M}$---similarly to the formalization of the arbitrary natural numbers in \citep{Horsten2019}. This guarantees that there is a canonical ``general'' arbitrary set, namely the identity function $\mathrm{id}_\mathcal{M} : M \to M; x \mapsto x$, whose scope is every set of $\mathcal{M}$. The model admits a natural semantics by evaluating sentences in the algebra of the power set of the state space associated to the arbitrary sets, which plays the same structural role that a Boolean algebra plays in standard Boolean-valued models. As such, the resulting underlying logic of sentences about arbitrary sets is supervaluationistic, aligning with Fine's view in \citep{Fine1985}.

The paper develops two related but importantly different semantic strategies.

\begin{itemize}[itemsep=0.4em]
\item In Section \ref{sec:one} we begin with a direct and
intensional proposal. We take all total functions $M \to M$
as our domain, representing each ground set $a \in M$ by
the constant function $i \mapsto a$, and interpret membership
as in the classical direct power:

\[
\text{for}\ f, g \in \prescript{M}{}{M}, f \in^{\mathcal{M}^*} g \quad \Longleftrightarrow \quad \forall i \in M\ (f(i) \in g(i)).
\]

\noindent
This clause captures the thought that $f$ is a member of $g$ only if the value of $f$ is a member of the value of $g$ in every state. The resulting structure validates several basic axioms, but it fails Extensionality and Pairing, and, more decisively, the full Separation scheme. The failure clarifies why the construction does not yield a full ZFC environment for arbitrary sets.

\item Section \ref{sec:two} contains the main model. We take the direct power $\mathfrak{M} \coloneqq \mathcal{M}^M$ and interpret \emph{both} membership and identity using a Boolean-valued semantics with truth values in $\mathcal{P}(M)$. Concretely, each formula $\varphi$ is assigned the set of states $i \in M$ in which $\varphi$ holds pointwise in $\mathcal{M}$; we declare $\varphi$ to be true in $\mathfrak{M}$ exactly when its truth value is all of $M$. A fundamental pointwise evaluation lemma then yields a transparent transfer of ZFC from $\mathcal{M}$ to $\mathfrak{M}$. Within this setting we develop the basic internal notions of \emph{ordinal} and \emph{cardinal}: arbitrary ordinals are precisely functions whose values are ordinals of $\mathcal{M}$, and arbitrary cardinals are likewise pointwise cardinals of $\mathcal{M}$. Classical linearity and minimality remain valid internally,
while global comparisons exhibit state-indexed behaviour: ZFC still yields trichotomy of cardinality as a disjunction, but distinct disjuncts may be witnessed at different states, and coordinatewise ``mixing'' produces many intermediate cardinals between two globally comparable endpoints.

\item In Section \ref{sec:three} we place these phenomena in a familiar model-theoretic frame. Choosing an ultrafilter on $M$ collapses $\mathcal{P}(M)$-truth values to two-valued truth and yields the corresponding ultrapower quotient; this explains why a mixture of finitely many fixed alternatives
collapses to one alternative in a two-valued reduction. We formalize this contrast using \emph{comparison profiles} for arbitrary ordinals and cardinals and show how profiles are collapsed by ultrafilter quotients. We then introduce a filtered subdomain (a ``cut'') that rigidifies objects below a chosen threshold cardinal while preserving variability above it, thereby controlling where intermediate-cardinal phenomena occur.


\item Finally, Section \ref{sec:four} develops a combinatorial interpretation of the functional models through the operator $X\mapsto X^\Omega$. This generalizes the assignment picture associated with $\mathcal{P}(X)\cong 2^X$: binary choices of membership give way to assignments of values across a fixed state space. We explain how the ground hierarchy stratifies the uniformly bounded arbitrary sets, distinguish this stratification from the construction of the full function domain, and compare our proposal with Button's iterative account. The filtered constructions of Section \ref{sec:three} provide a further way of distinguishing degrees of arbitrariness.
\end{itemize}

\section{Arbitrary sets}\label{sec:one}

There is a view according to which, besides individual objects, such as the Moon, a chair, or even the number two, there are also \emph{arbitrary} objects. They are the referents of the subjects of commonly used phrases such as ``a natural number is either even or odd,'' ``an arbitrary subset of an ordinal is well-ordered by membership,'' or ``an arbitrary man is human.''

Arbitrary objects have two distinguishing features. First, they may take other objects as values in different states \citep{Horsten2019}. For example, an arbitrary natural number may take the value $2$ in one state, $179$ in another, and $10^{12345}$ in yet another. Second, they behave in a distinctive way: their properties are determined by those shared by all the objects in the collection of values they may take, that is, by their \emph{range} \citep{Fine1985}. Thus, an arbitrary natural number whose possible values are all greater than $1$ has a unique factorization into primes, since every number in its range has that property. For a general arbitrary natural number, however, divisibility by $3$ does not obtain, since some of its possible values are not divisible by $3$. Its negation does not obtain either, since other possible values are divisible by $3$.

Given a collection $F$ of objects, an arbitrary $F$ is an object that assumes values in $F$. It is therefore natural to represent an arbitrary $F$ by a function from a given state space $\Omega$ into $F$. The existence of a general arbitrary $F$, that is, one whose range is all of $F$, depends on the cardinality of the state space: since each state determines a single value, the state space must be at least as large as $F$ if every object in $F$ is to be assumed as a value at least once. Indeed, \citep{Horsten2019} takes arbitrary natural numbers, for example, to be represented by the functions in $\prescript{\mathbb{N}}{}{\mathbb{N}}$, the set of all functions with the natural numbers as both domain and codomain. We shall follow this realization and take each function from a given collection $F$ into itself to represent a possible arbitrary $F$, with the collection $\prescript{F}{}{F}$ exhausting the representations of arbitrary $F$’s.

At this point, it is worth noting that arbitrary objects are not reduced to functions. Indeed, functions as described above merely represent the relation between an arbitrary object, the way it assumes its possible values, and how the values of different arbitrary objects interact. Nevertheless, from now onward, given a set-theoretic collection $M$, we shall speak of elements of $\prescript{M}{}{M}$ as if they were arbitrary sets, for the sake of brevity.

The aforementioned behavior of arbitrary objects is guided by what has become known, following Fine, as the \emph{Principle of Generic Attribution} (henceforth, \emph{PGA}). The principle says that an arbitrary $F$ has exactly those properties which are common to all the objects in its range. It may be better understood by reasoning about how the different values of arbitrary $F$'s behave: for example, an arbitrary set has the empty set as one of its subsets because in any state, the value it assumes is a superset of the empty set; likewise, an arbitrary number $a$ is double an arbitrary number $b$ when in every state of their shared state space, the value $a$ assumes is two times the value $b$ assumes. More generally, let $a_1, ..., a_n$ be arbitrary objects with the same state space, $\Omega$ be their state space, and $\mathrm{v}(a_i, w)$ be the value $a_i$ assumes in the state $w$. Then for any formula $\varphi$,

\begin{center}
$\varphi(a_1, ..., a_n)\ \Longleftrightarrow\ \forall w \in \Omega \varphi\big{(} \mathrm{v}(a_1, w), ..., \mathrm{v}(a_n, w) \big{)}$.
\end{center}

\noindent
For example, let $D(x, y)$ express that $x$ is divisible by $y$. Then, for an arbitrary number $a$, $D(a, 3)$ obtains iff for each state $w$ of the state space of $a$, $\mathrm{v}(a, w)$ is divisible by $3$---that is, $D \big{(} \mathrm{v}(a, w), 3 \big{)}$ obtains.

\begin{definition}\label{def:direct-power}
Let $\mathcal{A}$ be a first-order structure with domain $A$, and let $I$ be a nonempty index set. The direct power $\mathcal{A}^I = \prod_{i \in I}\mathcal{A}$ has as its domain the set
\[
A^I = \prescript{I}{}{A} = \{f \mid f : I \to A\}
\]
of all total functions from $I$ to $A$. For each $n$-ary relation symbol $R$, its interpretation is
\[
R^{\mathcal{A}^I}(f_1,\ldots,f_n)
\quad\Longleftrightarrow\quad
\forall i \in I\;
R^{\mathcal{A}}(f_1(i),\ldots,f_n(i)).
\]
Function symbols, when present, are interpreted coordinatewise, and each constant symbol $c$ is interpreted by the constant function $i \mapsto c^{\mathcal{A}}$. Equality is ordinary equality of functions: $f=g$ exactly when $f(i)=g(i)$ for every $i \in I$.
\end{definition}

We note, in the \emph{classical} direct power, compound formulas are evaluated using the usual two-valued first-order semantics, with quantifiers ranging over $A^I$. The coordinatewise characterization of atomic truth does not, in general, extend to arbitrary formulas.

Naturally, if $F$ carries a structure, then the representation of arbitrary $F$’s by functions in $\prescript{F}{}{F}$ allows that structure to be translated into the truth conditions of $F^F$, understood as a direct power of $F$. Of course, the faithfulness of formula evaluation in a direct power to the PGA depends on the choice of semantics. A classical direct power will not satisfy the PGA in its broadest scope. To see this, let $O(x)$ express that $x$ is odd, and $E(x)$ that $x$ is even. According to the PGA, a general arbitrary natural number $a$ will be such that $E(a) \vee O(a)$ holds, but neither $E(a)$, $\neg E(a)$, $O(a)$, nor $\neg O(a)$ holds. In a direct power, however, the truth of $E(a) \vee O(a)$ is compositional, and so reduces to the truth of one of its disjuncts. Hence the direct power does not satisfy $E(a) \vee O(a)$. If one wants to adopt the PGA in full generality, the semantics associated with structures of arbitrary objects must therefore be non-classical.\footnote{At least in a unilateral setting. As \citep{Yago2026} argues, in a bilateral setting---where asserting a sentence and rejecting it are not reducible one to another---, it may be possible for the PGA and classical logic to coexist. A bilateral analysis of the present subject is however outside the scope of this paper.} \citep{Horsten2019}, for example, drops the PGA in favor of classical logic---or, rather, takes the PGA to hold only as broadly as classical logic permits. On the other hand, \citep{Santambrogio} adopts intuitionistic logic, whereas \citep{Fine1985} takes supervaluationist logic to govern arbitrary objects.

To showcase the differences the choice of semantics makes, in this section we shall consider how much of ZF(C) a classical direct power can retain. We shall then contrast it, in the next section, with the preservation of a Boolean-valued direct power.

We work with \emph{set models} rather than proper classes, so that all domains are genuine sets in the metatheory. Let us thus fix a transitive set model $\mathcal{M} \models \mathrm{ZF}$ (or $\mathrm{ZFC}$ if desired) in the metatheory, and let the index set be $M$.

Define $\mathcal{M}^*$ as the direct power $\mathcal{M}^M$. As mentioned, elements of $\mathcal{M}^*$ are intended to represent arbitrary objects taking values in $\mathcal{M}$ in each state of their shared state space, which we take to be indexed by $M$.

In the direct power, the membership relation is interpreted as

\[
f \in^{\mathcal{M}^*} g \quad \Longleftrightarrow \quad \forall i \in M \big{(} f(i) \in g(i) \big{)}.
\]

\noindent
We also define $f \subseteq^{\mathcal{M}^*} g$ to abbreviate $\forall x (x \in^{\mathcal{M}^*} f \to x \in^{\mathcal{M}^*} g)$.

We assume Choice in the metatheory throughout this section. This does not require Choice to hold in the ground model $\mathcal{M}$, which may still be a model of ZF. For $a \in M$, write $f_a$ for the constant function $f_a(i)=a$ for every $i \in M$.



\begin{lemma}\label{lem:pointwise-subset}
For $f,g\in \mathcal{M}^*$,
\[
f\subseteq^{\mathcal{M}^*}g
\quad\Longleftrightarrow\quad
\bigl(\exists i\in M\;f(i)=\varnothing\bigr)
\ \vee\
\bigl(\forall i\in M\;f(i)\subseteq g(i)\bigr).
\]
In particular, if $f$ has a member in $\mathcal{M}^*$, then
$f\subseteq^{\mathcal{M}^*}g$ is equivalent to
$f(i)\subseteq g(i)$ for every $i\in M$.
\end{lemma}

\begin{proof}
If $f(i_0)=\varnothing$ for some $i_0\in M$, then $f$ has no member in $\mathcal{M}^*$, and $f\subseteq^{\mathcal{M}^*}g$ holds vacuously. If $f(i)\subseteq g(i)$ for all $i$, the definition of membership directly gives $f\subseteq^{\mathcal{M}^*}g$.

Conversely, suppose $f\subseteq^{\mathcal{M}^*}g$ and every $f(i)$ is nonempty. By Choice in the metatheory, choose $u\in M^M$ with $u(i)\in f(i)$ for every $i$. Here transitivity of $M$ ensures that each chosen value belongs to $M$. Fix $i\in M$ and $m\in f(i)$, and define
\[
u_{i,m}(j)=
\begin{cases}
m,&j=i,\\
u(j),&j\ne i.
\end{cases}
\]
Then $u_{i,m}\in^{\mathcal{M}^*}f$, so $u_{i,m}\in^{\mathcal{M}^*}g$ and $m\in g(i)$. Since $i$ and $m$ were arbitrary, all the coordinatewise inclusions hold. For the last assertion, a member of $f$ supplies $u$ directly, without any
further use of Choice.
\end{proof}

Given our model of arbitrary sets of $\mathcal{M}$ as the classical direct power, the question of what axioms of ZF(C) are satisfied by $\mathcal{M}^*$ may be answered by well-known results of preservation theory. The class of Horn sentences, for example, is inductively defined and preserved by direct powers \citep{Horn}, but does not exhaust the preserved fragment of first-order theory. In fact, the exact fragment preserved by direct products is offered in \citep{Weinstein}, but its syntactical characterization is not straightforward. Therefore, for the sake of clarity, we shall independently prove the preservation, or lack thereof, of each axiom.




\begin{proposition}
$\mathcal{M}^* \models \mathrm{Empty}\ \mathrm{Set}$.
\end{proposition}

\begin{proof}
Take $f_\varnothing \in M^M$. For any $g \in M^M$ and any $i \in M$,
$g(i) \not\in \varnothing=f_\varnothing(i)$. Thus
$g \not\in^{\mathcal{M}^*} f_\varnothing$ for every $g \in M^M$,
so $f_\varnothing$ witnesses the Empty Set axiom.
\end{proof}

\begin{proposition}
$\mathcal{M}^* \models \mathrm{Infinity}$.
\end{proposition}

\begin{proof}
Let $\omega \in M$ be the natural numbers of $\mathcal{M}$. We show $f_\omega$ is $\in^{\mathcal{M}^*}$-inductive. First, $f_\varnothing \in^{\mathcal{M}^*} f_\omega$ since $\varnothing \in \omega$ in $\mathcal{M}$. Next, suppose $g \in^{\mathcal{M}^*} f_\omega$. Define $s \in \mathcal{M}^*$ by $s(i) \coloneqq g(i) \cup \{g(i)\}$. Then $s(i) \in \omega$ for every $i$ by Infinity in $\mathcal{M}$, hence $s \in^{\mathcal{M}^*} f_\omega$.
\end{proof}

\begin{proposition}\label{prop:pairing}
$\mathcal{M}^* \not\models \mathrm{Pairing}$.
\end{proposition}

\begin{proof}
Suppose $f_\varnothing, f_\omega \in^{\mathcal{M}^*} h$. Then, there is $m \in \mathcal{M}^*$ with $f_\varnothing \neq m \neq f_\omega$ such that for each $i \in M$, $m(i) \in \{\varnothing, \omega\}$. But, by the definition of membership in $\mathcal{M}^*$, we have $m \in^{\mathcal{M}^*} h$. Thus, there is no set whose sole members are $f_\varnothing$ and $f_\omega$.
\end{proof}

\begin{proposition}\label{prop:pairing}
$\mathcal{M}^* \models \mathrm{Weak}\ \mathrm{Pairing}$.
\end{proposition}

\begin{proof}
Given $f, g \in \mathcal{M}^*$, define $h$ by $h(i) = \{f(i), g(i)\}$. Then $f(i) \in h(i)$ and $g(i) \in h(i)$ for all $i$.
\end{proof}

\begin{proposition}\label{prop:union}
$\mathcal{M}^* \models \mathrm{Union}$.
\end{proposition}

\begin{proof}
Fix $f \in \mathcal{M}^*$ and define $u$ by $u(i) \coloneqq \bigcup f(i)$. Union in $\mathcal{M}$ gives $u(i) \in M$ for every $i \in M$, so $u \in M^M$. We show that, for every $x \in M^M$,
\[
x \in^{\mathcal{M}^*} u
\quad\Longleftrightarrow\quad
\exists t\bigl(t \in^{\mathcal{M}^*} f
\wedge x \in^{\mathcal{M}^*} t\bigr).
\]
($\Leftarrow$) If $t \in^{\mathcal{M}^*} f$ and $x \in^{\mathcal{M}^*} t$, then $t(i) \in f(i)$ and $x(i) \in t(i)$ for every $i \in M$. Hence $x(i) \in \bigcup f(i)=u(i)$ for every $i$, so $x \in^{\mathcal{M}^*} u$.

($\Rightarrow$) If $x \in^{\mathcal{M}^*} u$, then for each $i \in M$ there is $t_i \in f(i)$ with $x(i) \in t_i$. Using Choice in the metatheory, define $t \in M^M$ by $t(i)=t_i$. Then $t \in^{\mathcal{M}^*} f$ and $x \in^{\mathcal{M}^*} t$.
\end{proof}


\begin{proposition}\label{prop:powerset}
$\mathcal{M}^* \not\models \mathrm{Powerset}$.
\end{proposition}

\begin{proof}
Fix $f\in M^M$. Suppose there were $p\in M^M$ such that
\[
\forall x\in M^M\;
\bigl(x\subseteq^{\mathcal{M}^*}f
\longrightarrow x\in^{\mathcal{M}^*}p\bigr).
\]
Since $M$ has at least two elements, choose distinct $i_0,i_1\in M$. Define $x\in M^M$ by
\[
x(i)=
\begin{cases}
p(i_0), & i=i_0,\\
\varnothing, & i\ne i_0.
\end{cases}
\]
Because $x(i_1)=\varnothing$, no $y\in M^M$ satisfies $y\in^{\mathcal{M}^*}x$. Consequently, $x\subseteq^{\mathcal{M}^*}f$ holds vacuously, so our assumption yields $x\in^{\mathcal{M}^*}p$. Evaluating this membership at $i_0$ gives
\[
p(i_0)=x(i_0)\in p(i_0),
\]
contradicting Foundation in the ground model $\mathcal{M}$. Thus no object contains all the internal subsets of $f$; in particular, the Powerset axiom fails.
\end{proof}



\begin{proposition}
$\mathcal{M}^* \models \mathrm{Foundation}$.
\end{proposition}


\begin{proof}
Let $f \in \mathcal{M}^*$ have a member $h$, and fix $i_0 \in M$. Since $h(i_0) \in f(i_0)$, the set $f(i_0)$ is nonempty. Foundation in $\mathcal{M}$ gives $a \in f(i_0)$ such that $a \cap f(i_0)=\varnothing$. Define $g \in M^M$ by
\[
g(i)=
\begin{cases}
a,&i=i_0,\\
h(i),&i\neq i_0.
\end{cases}
\]
Then $g(i) \in f(i)$ for every $i \in M$, so $g \in^{\mathcal{M}^*} f$. If some $t \in M^M$ were a member of both $f$ and $g$, we would have
\[
t(i_0) \in f(i_0) \cap g(i_0)
=f(i_0)\cap a=\varnothing,
\]
a contradiction. Thus $g$ witnesses Foundation for $f$.
\end{proof}

\begin{proposition}\label{prop:ext-fails}
$\mathcal{M}^* \not\models \mathrm{Extensionality}$.
\end{proposition}


\begin{proof}
Choose distinct $i_0,i_1 \in M$ and define $g \in M^M$ by
\[
g(i)=
\begin{cases}
\{\varnothing\},&i=i_0,\\
\varnothing,&i\neq i_0.
\end{cases}
\]
Then $g \neq f_\varnothing$, since $g(i_0)=\{\varnothing\}$. Nevertheless, neither object has a member in $\mathcal{M}^*$: $f_\varnothing$ is empty at every coordinate, and $g(i_1)=\varnothing$. Thus they have exactly the same members but are distinct, contradicting Extensionality.
\end{proof}

\begin{proposition}\label{prop:sep-fails}
$\mathcal{M}^* \not\models \mathrm{Separation}$.
\end{proposition}

\begin{proof}
Just take the separation scheme defining the pairing of two members of a given $f \in \mathcal{M}^*$. As we have seen above, one of the directions of the biconditional of the scheme will fail.
\end{proof}

Call an $\mathcal{L}_{\in}$-formula $\varphi( x)$ \emph{inheritable relative to $\mathcal{M}$} if, for every $ f \in M^M$,
\[
\forall i\in M\;
\mathcal{M}\models\varphi( f(i))
~\Longrightarrow ~
\mathcal{M}^*\models\varphi( f).
\]
An explicit syntactic sufficient condition for inheritability is given by the \emph{positive Horn fragment}. This is the fragment generated from atomic formulas using conjunction and the two quantifiers:
\[
\psi\;:=\;x=y\;\mid\;x\in y\;\mid\;
(\psi\wedge\psi)\;\mid\;
\exists x\,\psi\;\mid\;\forall x\,\psi.
\]
Bounded quantifiers are treated as abbreviations and must be expanded before applying this syntactic criterion. For this fragment we have transfer in both directions, as follows.

\begin{lemma}\label{lem:positive-horn-pointwise}
For every positive Horn formula $\psi( x)$ and  $ f\in M^M$
\[
\mathcal{M}^*\models\psi( f)  \Longleftrightarrow 
\forall i\in M\; \mathcal{M}\models\psi( f(i)).
\] 
In particular, every positive Horn formula is inheritable.
\end{lemma}

\begin{proof}
We argue by induction on the complexity of $\psi$, uniformly allowing finitely many named parameters. Each parameter $h\in M^M$ is interpreted by $h$ in $\mathcal{M}^*$ and by $h(i)$ in $\mathcal{M}$ at index $i$. Thus, when applying the induction hypothesis to $\theta(g,x)$, we regard $g$ as a parameter. The atomic cases follow from the definition of membership in $\mathcal{M}^*$ and the equality of functions; the conjunction case follows immediately from the induction hypothesis.

Suppose $\psi(x)=\exists y\,\theta(y,x)$. If $\mathcal{M}^*\models\psi(f)$, take $g\in M^M$ such that $\mathcal{M}^*\models\theta(g,f)$. By the induction hypothesis, $\mathcal{M}\models\theta(g(i),f(i))$ for every $i\in M$, so $\mathcal{M}\models\exists y\,\theta(y,f(i))$ for every $i\in M$. Conversely, suppose the latter statement holds. For each $i\in M$, choose $m_i\in M$ such that $\mathcal{M}\models\theta(m_i,f(i))$. Using Choice in the metatheory, define $g\in M^M$ by $g(i)=m_i$. The induction hypothesis gives $\mathcal{M}^*\models\theta(g,f)$, hence $\mathcal{M}^*\models\psi(f)$.

Suppose $\psi(x)=\forall y\,\theta(y,x)$. If $\mathcal{M}^*\models\psi(f)$, fix $i\in M$ and $m\in M$, and let $f_m$ be the constant function with value $m$. Then $\mathcal{M}^*\models\theta(f_m,f)$, so the induction hypothesis yields $\mathcal{M}\models\theta(m,f(i))$. Since $i$ and $m$ were arbitrary, $\mathcal{M}\models\forall y\,\theta(y,f(i))$ for every $i\in M$. Conversely, suppose this holds and take any $g\in M^M$. Then $\mathcal{M}\models\theta(g(i),f(i))$ for every $i\in M$. By the induction hypothesis, $\mathcal{M}^*\models\theta(g,f)$. Since $g$ was arbitrary, $\mathcal{M}^*\models\psi(f)$.
\end{proof}

\begin{proposition}\label{prop:positive-horn-separation}
 $\mathcal{M}^*$ satisfies
$\mathrm{Separation}(\varphi)$ for every positive Horn formula
$\varphi(x)$.
\end{proposition}

\begin{proof}
Let $\varphi(x)$ be positive Horn and fix $f\in M^M$. Define $g$ by 
\[
g(i)=\{y\in f(i):\mathcal{M}\models\varphi(y)\}
~ (i\in M).
\]
Separation in $\mathcal{M}$ gives $g(i)\in M$ for every $i\in M$, so $g\in M^M$. If $h\in^{\mathcal{M}^*}g$, then $h(i)\in f(i)$ and $\mathcal{M}\models\varphi(h(i))$ for every $i\in M$. Thus $h\in^{\mathcal{M}^*}f$ and, by Lemma~\ref{lem:positive-horn-pointwise}, $\mathcal{M}^*\models\varphi(h)$.

Conversely, suppose $h\in^{\mathcal{M}^*}f$ and $\mathcal{M}^*\models\varphi(h)$. Then $h(i)\in f(i)$ for every $i\in M$, and Lemma~\ref{lem:positive-horn-pointwise} gives $\mathcal{M}\models\varphi(h(i))$ for every $i\in M$. Hence $h(i)\in g(i)$ for every $i\in M$, so $h\in^{\mathcal{M}^*}g$. Therefore, for every $h\in M^M$,
\[
h\in^{\mathcal{M}^*}g
\quad\Longleftrightarrow\quad
h\in^{\mathcal{M}^*}f
\quad\text{and}\quad
\mathcal{M}^*\models\varphi(h).
\]
Thus $g$ witnesses the required instance of Separation.
\end{proof}

\begin{proposition}\label{prop:collection}
 $\mathcal{M}^* \models \mathrm{Collection}$.
\end{proposition}

\begin{proof}
Fix $\varphi(y,z)$ and let $f \in \mathcal{M}^*$ be such that $\mathcal{M}^* \models \forall y(y \in f \to \exists z\,\varphi(y,z)).$ By the Feferman--Vaught theorem for direct products \citep[Section~6.3]{ChangKeisler1990}, there are finitely many parameter-free formulas $\theta_1(y,z),\ldots,\theta_n(y,z)$ such that the truth of $\varphi(g,h)$ in $\mathcal{M}^*$ is determined by the sets
\[
\{i \in M \mid
\mathcal{M} \models \theta_j(g(i),h(i))\},
\qquad 1 \leq j \leq n.
\]
Thus, replacing $h$ while preserving these sets preserves the truth of $\varphi(g,h)$.

Let $J=\{1,\ldots,n\}$. For each $S \subseteq J$, define
\[
\chi_S(y,z) \coloneqq
\bigwedge_{j \in S}\theta_j(y,z)
\wedge
\bigwedge_{j \in J \setminus S}\neg\theta_j(y,z).
\]
Fix $i \in M$. Separation in $\mathcal{M}$ gives $A_{i,S} \coloneqq \{a \in f(i) \mid \mathcal{M} \models \exists b\,\chi_S(a,b)\} \in M.$ By Collection in $\mathcal{M}$, there is $B_{i,S} \in M$ such that, for every $a \in A_{i,S}$,
\[
\mathcal{M} \models
\exists b \in B_{i,S}\,\chi_S(a,b).
\]
Using Choice in the metatheory, choose these sets for every $i \in M$ and $S \subseteq J$, and define $w$ by 
\[
w(i) \coloneqq \bigcup_{S \subseteq J} B_{i,S}.
\]
Since $J$ is finite, $w(i) \in M$ for every $i \in M$, so $w \in M^M$.

Now let $g \in^{\mathcal{M}^*} f$. By our assumption, there is $h \in \mathcal{M}^*$ such that $\mathcal{M}^* \models \varphi(g,h)$. For each $i \in M$, let
\[
S_i \coloneqq
\{j \in J \mid
\mathcal{M} \models \theta_j(g(i),h(i))\}.
\]
Since $g(i) \in f(i)$ and $h(i)$ witnesses $\exists b\,\chi_{S_i}(g(i),b)$ in $\mathcal{M}$, we have $g(i) \in A_{i,S_i}$. Hence there is $b_i \in B_{i,S_i}$ such that $\mathcal{M} \models \chi_{S_i}(g(i),b_i).$ Using Choice in the metatheory, define $u \in M^M$ by $u(i)=b_i$. Then $u(i) \in w(i)$ for every $i \in M$, so $u \in^{\mathcal{M}^*} w$. Moreover, by the definition of $\chi_{S_i}$, for every $i \in M$ and $j \in J$,
\[
\mathcal{M} \models \theta_j(g(i),u(i))
\quad\Longleftrightarrow\quad
\mathcal{M} \models \theta_j(g(i),h(i)).
\]
The determining sets are therefore unchanged, and the Feferman--Vaught theorem yields $\mathcal{M}^* \models \varphi(g,u)$. Since $g$ was arbitrary, $\mathcal{M}^* \models \forall v(v \in f \to \exists u(u \in w \wedge \varphi(v,u))).$ Thus $w$ witnesses the required instance of Collection.
\end{proof}

Thus, Proposition \ref{prop:sep-fails} already blocks the claim that $\mathcal{M}^*$ satisfies ZFC—even with Extensionality set aside—since the full range of the Separation schema is unavailable. Nevertheless, Proposition~\ref{prop:positive-horn-separation} shows that Separation holds for positive Horn formulas.

As Fine would argue, the failure of Separation reflects a deeper incompatibility between classical compositional semantics and the full generality of the PGA. In $\mathcal{M}^*$, the truth condition for $\varphi \vee \psi$ is evaluated classically: $\mathcal{M}^* \models \varphi \vee \psi$ iff either $\mathcal{M}^* \models \varphi$ or $\mathcal{M}^* \models \psi$. But the broadest application of the PGA requires that a disjunction $\varphi(g) \vee \psi(g)$ concerning an arbitrary object $a$---where $g$ represents $a$---hold whenever $\varphi(g(i)) \vee \psi(g(i))$ holds at every state $i$, even if no single disjunct holds at all states simultaneously.


For example, let $F(x)$ express that $x$ is a finite ordinal, and let $L(x)$ express that $x$ is an infinite ordinal. Let $g \in M^M$ be a function whose range consists only of ordinals of $\mathcal{M}$ and contains at least one finite ordinal and at least one infinite ordinal. Then every value of $g$ is either finite or infinite, so
\[
\forall i \in M \bigl(F(g(i)) \vee L(g(i))\bigr).
\]
The PGA therefore requires $F(g) \vee L(g)$ to hold. In $\mathcal{M}^*$, however, this disjunction would require either
\[
\forall i \in M\, F(g(i)),
\]
that is, that every value assumed by $g$ be finite, or
\[
\forall i \in M\, L(g(i)),
\]
that is, that every value assumed by $g$ be infinite. Both conditions are false, since the range of $g$ contains both a finite and an infinite ordinal. The classical semantics of $\mathcal{M}^*$ is therefore incompatible with the unrestricted application of the PGA.

Thus, with a purely recursive definition of the extended membership relation, we end up with a structure in which arbitrary sets behave much like sets---except where it matters most. This is not entirely surprising: arbitrary sets, as interpreted here, are intensional in nature, meaning they are individuated by the way they are generated, not merely by the collection of their members. As a result, the Extensionality axiom fails in this setting. 
In what follows, we explore an alternative approach that modifies the semantics to better align with ordinary set-theoretic behavior.

\section{Models of functions in $\mathcal{M}$}\label{sec:two}

In this section, we introduce a semantics that departs from the one considered earlier. Let $\mathfrak{M} = \mathcal{M}^M$, that is, the direct power of $\mathcal{M}$ by the index set $M$, as in $\mathcal{M}^*$. We shall now, however, define a different assignment of truth values to formulas.\footnote{Notice $\mathcal{M}^*$ and $\mathfrak{M}$ have the same domain. We use different names for each model to stress the difference in their semantics.}

Let the Boolean algebra of truth values be $\mathcal{P}(M)$. We define the following Boolean evaluation of $\mathfrak{M}$.

\begin{definition}\label{def:truth-values}
For each $\mathcal{L}_\in$-formula $\varphi$ with parameters from $\mathfrak{M}$, define a truth value $\llbracket \varphi \rrbracket \subseteq M$ by recursion:

\begin{align*}
\llbracket x \in y \rrbracket & = \{i \in M \mid x(i) \in y(i)\},\\
\llbracket x = y \rrbracket &  = \{i \in M \mid x(i) = y(i)\},\\
\llbracket \neg \varphi \rrbracket &  = M \setminus \llbracket \varphi \rrbracket,\\
\llbracket \varphi \wedge \psi \rrbracket &  =\llbracket \varphi \rrbracket \cap \llbracket \psi \rrbracket,\\
\llbracket \varphi \vee \psi \rrbracket &  = \llbracket \varphi \rrbracket \cup \llbracket \psi \rrbracket,\\
\llbracket \varphi \to \psi \rrbracket &  = (M \setminus \llbracket \varphi \rrbracket) \cup \llbracket \psi \rrbracket,\\
\llbracket \exists x \varphi(x) \rrbracket &  = \bigcup_{u\in\mathfrak{M}} \llbracket \varphi(u) \rrbracket,\\
\llbracket \forall x \varphi(x) \rrbracket &  = \bigcap_{u\in\mathfrak{M}} \llbracket \varphi(u) \rrbracket.
\end{align*}

\noindent
We write $\mathfrak{M} \models \varphi$ iff $\llbracket \varphi \rrbracket = M$.
\end{definition}

Notice $\mathfrak{M}$ is thus a $\mathcal{P}(M)$-valued model of ZF(C), in the contemporary sense (\citep{Mansfield1971}, \citep{Hamkins2012}, \citep{Wu2024}).

We may also note that the above semantics is supervaluationistic, in accordance with Fine’s view. In fact, it allows for the desired non-compositionality of truth.

\begin{proposition}\label{prop:fundamental}
For every $\mathcal{L}_\in$-formula $\varphi(x_1,\dots,x_n)$ and $a_1, ...,a_n \in \mathfrak{M}$,

\[
\llbracket\varphi(a_1,\dots,a_n)\rrbracket
=\{i\in M \mid \mathcal{M} \models \varphi(a_1(i),\dots,a_n(i))\}.
\]
\end{proposition}

\begin{proof}
By induction on the complexity of $\varphi$. Atomic and Boolean cases are immediate from Definition \ref{def:truth-values}. For quantifiers, e.g. $\exists x\,\psi(x,\vec a)$, we have
\[
\llbracket\exists x\,\psi(x,\vec a)\rrbracket
=\bigcup_{u\in\mathfrak{M}}\llbracket\psi(u,\vec a)\rrbracket.
\]
By the induction hypothesis, $i$ lies in the union iff there exists $u\in\mathfrak{M}$ with $M\models\psi(u(i),\vec a(i))$, which is equivalent to $M\models\exists x\,\psi(x,\vec a(i))$. The universal quantifier case is analogous.
\end{proof}

Since, in the defined semantics, a sentence is true just in case its assigned truth value is the top element $M$, we may notice the choice of $\mathcal{P}(M)$ as the designated Boolean algebra for the semantics is entirely natural by considering the PGA. Indeed, the above result shows how $\mathfrak{M}$ makes the PGA true: for $f_1, ..., f_n \in \mathfrak{M}$ and a formula $\varphi$, by the above definition and Proposition \ref{prop:fundamental}, $\varphi(f_1, ..., f_n)$ obtains in $\mathfrak{M}$ precisely when $\big{\{} i \in M \mid \varphi \big{(} f_1(i), ..., f_n(i) \big{)} \big{\}}$ is the top element of the Boolean algebra, and therefore just in case for every $i \in M$, $\varphi \big{(} f_1(i), ..., f_n(i) \big{)}$ obtains in $\mathcal{M}$. Given our observation, in Section \ref{sec:one}, of the representation of arbitrary objects by functions in the direct power, that can be translated to the fact $\varphi(a_1, ..., a_n)$ holds of the arbitrary sets $a_1, ..., a_n$ just in case for each state $w$ of their state space, $\varphi \big{(} \mathrm{v}(a_1, w), ..., \mathrm{v}(a_n, w) \big{)}$ is the case.

\begin{theorem}\label{thm:diagonal}
Define $\delta : M \to \prescript{M}{}{M}$ by $\delta(a)(i) = a$ for all $i\in M$. Then for every $\mathcal{L}_\in$-formula $\varphi(\vec{x})$ and $\vec{a} \in M$,

\[
\mathcal{M} \models \varphi(\vec{a}) \quad \Longleftrightarrow \quad \mathfrak{M} \models \varphi(\delta(\vec{a})).
\]
\end{theorem}

\begin{proof}
Apply Proposition \ref{prop:fundamental} to $\varphi(\delta(\vec{a}))$.
\end{proof}

\begin{corollary}\label{cor:zfc-transfer}
If $\mathcal{M} \models \mathrm{ZFC}$ then $\mathfrak{M} \models \mathrm{ZFC}$.
\end{corollary}

\begin{proof}
Let $\varphi$ be any axiom (or instance of a scheme) of ZFC. Then $\varphi$ is a sentence, so Proposition \ref{prop:fundamental} yields

\[
\llbracket \varphi \rrbracket=\{i \in M \mid \mathcal{M} \models \varphi\}.
\]

\noindent
Since $\mathcal{M} \models \varphi$, this set is all of $M$, hence $\mathfrak{M} \models \varphi$.
\end{proof}

\subsection{Pointwise definable operations and two notions of ``between''}\label{subsec:card-arith}

From now on, assume that the ground model $\mathcal{M}$ satisfies $\mathrm{ZFC}$. By Corollary~\ref{cor:zfc-transfer}, $\mathfrak{M}\models\mathrm{ZFC}$, so all the standard constructions on ordinals and cardinals (successor, power set, cardinality, exponentiation) are available \emph{internally} in $\mathfrak{M}$. In our $\mathcal{P}(M)$-valued semantics these constructions are computed \emph{pointwise}.

\begin{proposition}\label{prop:definable-pointwise}
Let $\varphi(x,y)$ be an $\mathcal{L}_\in$-formula such that $\mathrm{ZFC}$ proves: for every $x$ there is a unique $y$ with $\varphi(x,y)$. If $f, g \in \mathfrak{M}$ satisfy $\mathfrak{M} \models \varphi(f,g)$, then for all $i \in M$ we have $\mathcal{M} \models \varphi(f(i), g(i))$. In particular, writing $\mathcal{P}^{\mathfrak{M}}(f)$ for the internal powerset of $f$ in $\mathfrak{M}$, and $\mathrm{card}^{\mathfrak{M}}(f)$ for the (internal) cardinal of $f$ in $\mathfrak{M}$, we have for every $f \in \mathfrak{M}$:

\[
\mathcal{P}^{\mathfrak{M}}(f)(i) = \mathcal{P}^\mathcal{M}(f(i)) \mbox{ and }\mathrm{card}^{\mathfrak{M}}(f)(i) = \mathrm{card}^\mathcal{M}(f(i)).
\]

\noindent
Moreover, for cardinals $\kappa, \lambda \in \mathbf{Card}
^{\mathfrak{M}}$,

\[
(\kappa^{+ \mathfrak{M}})(i) = (\kappa(i))^{+M},
\quad
(2^{\kappa, \mathfrak{M}})(i) = 2^{\kappa(i)}{}^{M},
\quad
(\kappa^{\lambda, \mathfrak{M}})(i) = (\kappa(i))^{\lambda(i)}{}^{M}.
\]
\end{proposition}

\begin{proof}
If $\mathfrak{M} \models \varphi(f, g)$ then $\llbracket \varphi(f, g) \rrbracket = M$, and Proposition \ref{prop:fundamental} yields $\forall i \in M \big{(} \mathcal{M} \models \varphi(f(i), g(i)) \big{)}$. The displayed equalities follow by taking $\varphi$ to define the relevant operations in ZFC (powerset, cardinality, successor, exponentiation).
\end{proof}

For the discussion of crisp intermediates in Section~\ref{sec:three}, we introduce the following pointwise notation. For $\kappa \in \mathbf{Card}^{\mathfrak{M}}$, let
\[(2^{\kappa,\bullet})(i)
\coloneqq \bigl(2^{\kappa(i)}\bigr)^{\mathcal{M}}
\qquad (i \in M).
\]
By Proposition~\ref{prop:definable-pointwise}, this agrees with the internal cardinal exponentiation $2^{\kappa,\mathfrak{M}}$. The notation emphasizes its coordinatewise computation and will be used in Propositions~\ref{prop:no-mix-below} and~\ref{prop:mix-above}.

What changes relative to an ordinary two-valued model is not the arithmetic itself but the \emph{shape of comparisons}. For instance, for cardinals $\kappa, \lambda$ the truth value $\llbracket \kappa \leq \lambda \rrbracket$ records \emph{where} the inequality holds. ZFC proves linearity of cardinals as a disjunction, so in our semantics one typically has

\[
\llbracket \kappa\le\lambda\rrbracket\cup \llbracket \lambda\le\kappa\rrbracket = M
\]

\noindent
even when neither truth value is all of $M$.




Finally, constant cardinals already exhibit the intended intensionality: for $\kappa \in \mathbf{Card}^{\mathcal{M}}$, $\mathcal{P}^{\mathfrak{M}}(\delta(\kappa))$ is the constant function $\delta(\mathcal{P}^{\mathcal{M}}(\kappa))$, and its members in $\mathfrak{M}$ are precisely the functions $A \in \mathfrak{M}$ satisfying $A(i) \subseteq \kappa$ for every $i \in M$. Thus, external multiplicity of pointwise subsets does not translate into larger internal cardinalities, but it does generate many distinct arbitrary sets occupying the same coordinatewise scope.

\subsection{The cumulative hierarchy of $\mathfrak{M}$}\label{subsection cumulative hierarchy}

Despite the differences with usual models of ZFC, and the elements of $\mathfrak{M}$ being externally functions, there is a hierarchical conception behind them, which allows us to draw an intuitive definition of its levels. We shall follow a definition analogous to the standard cumulative hierarchy. 


For our fixed transitive ground model $\mathcal{M}$, put $\theta=M\cap\mathbf{On}$ and let $\mathcal{M}_\alpha$ denote its $\alpha$th cumulative level, for $\alpha<\theta$. Thus 
\[
\begin{aligned}
\mathcal{M}_0&=\varnothing,\\
\mathcal{M}_{\alpha+1}
  &=\mathcal{P}^{\mathcal{M}}(\mathcal{M}_\alpha),\\
\mathcal{M}_\lambda&=\bigcup_{\alpha<\lambda}\mathcal{M}_\alpha
       \qquad(\lambda<\theta\text{ a limit ordinal}),\\
M&=\bigcup_{\alpha<\theta}\mathcal{M}_\alpha.
\end{aligned}
\]
Here and below, indices of ground-model levels are less than $\theta$, and powersets in the ground hierarchy are computed in $\mathcal{M}$.

Following this construction, in building the model of arbitrary sets, we start by analogously taking an arbitrary empty set as the $0$th level of the hierarchy. Within arbitrary object theory, an arbitrary empty set may naturally be taken to be an arbitrary object whose values are all empty sets. But that produces the constant arbitrary object with value the empty set in every state. Denote this object by $\delta(\varnothing)$. We may then take each successor level of the hierarchy to be built by taking arbitrary subsets of the preceding level, and limit levels as the union of all preceding levels. In that way, for example, the $1$st level is composed by collecting arbitrary subsets of the $0$th level, $\delta(\varnothing)$. Its only internal subset is $\delta(\varnothing)$, and its internal powerset is therefore $\delta(\{\varnothing\})$. The $2$nd level, in its turn, is likewise composed by collecting arbitrary subsets of the $1$st level. Here, we find the first explosion of the external cardinalities of the levels compared to the classical von Neumann hierarchy. For both $\varnothing$ and $\{\varnothing\}$ are subsets of the latter, so since the subset relation is evaluated pointwise, any arbitrary set assuming values in the set $\big{\{} \varnothing, \{\varnothing\} \big{\}}$ is an arbitrary subset of $\delta(\{\varnothing\})$. Thus the external collection of globally true members of the
$2$nd level has cardinality $2^{|M|}$.

Given the pointwise interpretation of functions for arbitrary sets, as previously noted, an arbitrary set $f$ is a subset of an arbitrary set $g$ when in each state of the state space, the value of $f$ is a subset of the value of $g$. Recall from our previous observations that this means the power set of an arbitrary set $f$ is the arbitrary set which assumes as a value the power set of the value of $f$ in every state of the state space. Thus, the construction of the hierarchy follows the same pattern as von Neumann's, in that one starts with the arbitrary empty set as the $0$th level, takes the (arbitrary) powerset of the preceding level to compose the successor level, and takes (arbitrary) unions at limit levels, since collecting arbitrary subsets inside $\mathfrak{M}$ of a given level $\alpha$ of the hierarchy of $\mathfrak{M}$ is to simply take the (arbitrary) power set of the $\alpha$th level.

More precisely, let $\mathfrak{M}_\alpha$ be the $\alpha$th level of the hierarchy of $\mathfrak{M}$. Then, we may see:\[
\begin{aligned}
\mathfrak{M}_0
  &= \delta(\varnothing),\\
\mathfrak{M}_{\alpha+1}
  &= \mathcal{P}^{\mathfrak{M}}(\mathfrak{M}_\alpha)
   = \delta(\mathcal{M}_{\alpha+1}),\\
\mathfrak{M}_\lambda
  &= \bigcup_{\alpha < \lambda}^{\mathfrak{M}} \mathfrak{M}_\alpha
   = \delta(\mathcal{M}_\lambda),
\end{aligned}
\]

\noindent
where $\delta(\mathcal{M}_\alpha)$ is the constant arbitrary set assuming value $\mathcal{M}_\alpha$.

We may now observe that, given the above construction of the levels of the cumulative hierarchy of $\mathfrak{M}$, there is a correspondence between the elements of the $\alpha$th level of the hierarchy and what the usual conception of arbitrary objects would tell us the arbitrary sets of the level $\alpha$ would be. 

Intuitively, an arbitrary set of the $\alpha$th level of the cumulative hierarchy of $\mathcal{M}$ is an arbitrary object assuming values inside $\mathcal{M}_\alpha$---that is, such that in every state of the state space its value is an element of $\mathcal{M}_\alpha$. In general, that tells us $f$ is an arbitrary set of $\mathcal{M}_\alpha$ when $\forall i \in M \big{(} f(i) \in \mathcal{M}_\alpha \big{)}$. 

Given our previous exposition of the adapted von Neumann construction of $\mathfrak{M}$, we may notice $f$ is in $\mathfrak{M}_{\alpha+1}$ iff it is in the arbitrary power set of $\mathfrak{M}_\alpha$, iff it is an arbitrary subset of $\mathfrak{M}_\alpha$, iff in every state the value of $f$ is a subset of the value of $\mathfrak{M}_\alpha = \delta(\mathcal{M}_\alpha)$, iff in every state the value $f$ assumes is in $\mathcal{M}_{\alpha+1}$. Thus,
\[
\begin{aligned}
f \in^{\mathfrak{M}} \mathfrak{M}_{\alpha+1}
  &\Longleftrightarrow
    f \in^{\mathfrak{M}} \mathcal{P}^{\mathfrak{M}}(\mathfrak{M}_\alpha)\\
  &\Longleftrightarrow
    \forall i \in M\,
    \bigl(f(i) \in \mathcal{P}(\mathcal{M}_\alpha)\bigr)\\
  &\Longleftrightarrow
    \forall i \in M\,
    \bigl(f(i) \in \mathcal{M}_{\alpha+1}\bigr),\\[0.5em]
f \in^{\mathfrak{M}} \mathfrak{M}_\lambda
  &\Longleftrightarrow
    \forall i \in M\,
    \bigl(f(i) \in \bigcup_{\alpha < \lambda } \mathcal{M}_\alpha\bigr),
    \quad \text{for limit } \lambda\\
  &\Longleftrightarrow
    \forall i \in M\,
    \bigl(f(i) \in \mathcal{M}_\lambda\bigr).
\end{aligned}
\]





\noindent


At this point, we may distinguish the internal hierarchy objects from their external collections of globally true members. For $\alpha<\theta$, let
\[
V_\alpha^{\mathfrak{M}}
 :=\{f\in M^M\mid\mathfrak{M}\models f\in\mathfrak{M}_\alpha\}
 =\{f\in M^M\mid\forall i\in M\,
                         f(i)\in\mathcal{M}_\alpha\}.
\]
Thus $V_0^{\mathfrak{M}}=\varnothing$, and
\[
V_{\alpha+1}^{\mathfrak{M}}
 =\{f\in M^M\mid\forall i\in M\,
            f(i)\in\mathcal{P}^{\mathcal{M}}(\mathcal{M}_\alpha)\}.
\]
For a limit ordinal $\lambda<\theta$, however, we have only
\[
\bigcup_{\alpha<\lambda}V_\alpha^{\mathfrak{M}}
 \subseteq V_\lambda^{\mathfrak{M}},
\]
and the inclusion can be strict. For example, define $f(i)=i$ when $i\in\omega$, and $f(i)=\varnothing$ otherwise. Then $f\in V_\omega^{\mathfrak{M}}$, but $f\notin V_n^{\mathfrak{M}}$ for any $n<\omega$, since $f(n)=n\notin\mathcal{M}_n$. The internal union at a limit stage therefore need not give the external union of the collections of globally true members. 

Each $V_\alpha^{\mathfrak{M}}$ has external cardinality $|\mathcal{M}_\alpha|^{|M|}$. In particular, for $1<\alpha<\theta$ this is $2^{|M|}$, since $2\leq|\mathcal{M}_\alpha|\leq|M|$. This external collection should be distinguished from the single internal object $\mathfrak{M}_\alpha=\delta(\mathcal{M}_\alpha)$. Both hierarchies are increasing: for $\alpha\leq\beta<\theta$,

\begin{center}
$\mathfrak{M}\models
  \mathfrak{M}_\alpha\subseteq\mathfrak{M}_\beta,
\qquad
V_\alpha^{\mathfrak{M}}\subseteq V_\beta^{\mathfrak{M}}$.
\end{center}

\noindent
The external hierarchy need not be continuous at limits.

Nevertheless, even the union of all these external levels does not exhaust the function domain. Put
\[
V^{\mathfrak{M}}
 :=\bigcup_{\alpha<\theta}V_\alpha^{\mathfrak{M}}
 \subsetneq M^M.
\]
Indeed, $\mathrm{id}_M\in M^M$, whereas for every $\alpha<\theta$ we have $\mathrm{id}_M(\alpha)=\alpha\notin\mathcal{M}_\alpha$. Thus $V^{\mathfrak{M}}$ consists precisely of the arbitrary sets whose values admit a uniform bound in the ground hierarchy.

\subsection{Arbitrary ordinals and arbitrary cardinals}

Let $\mathrm{Ord}(x)$ be the usual formula expressing that $x$ is an ordinal. The arbitrary ordinals are defined internally:
\[
\begin{aligned}
\mathbf{On}^{\mathfrak{M}}
 &:=\{f\in\mathfrak{M}\mid
               \mathfrak{M}\models\mathrm{Ord}(f)\}\\
 &=\{f\in M^M\mid\forall i\in M\,
               \mathcal{M}\models\mathrm{Ord}(f(i))\}\\
 &=\{f\in M^M\mid
               \operatorname{ran}(f)\subseteq\mathbf{On}^{\mathcal{M}}\},
\end{aligned}
\]
where the second equality follows from Proposition~\ref{prop:fundamental}. In particular, the logical conditions defining an ordinal are evaluated within $\mathfrak{M}$, before taking global truth.

Thus $f$ is an ordinal of $\mathfrak{M}$ exactly when every value of $f$ is an ordinal of $\mathcal{M}$. Internally, the usual ordinal trichotomy remains valid. Externally, the relation $\mathfrak{M}\models f\in g$ need not linearly order $\mathbf{On}^{\mathfrak{M}}$. For example, take $f,g\in\mathfrak{M}$
such that for every $i\in M$,

\medskip

\begin{minipage}{0.45\textwidth}
\begin{center}
    $f(i) = \begin{cases}
        \varnothing,\  i = \varnothing\\
        \{\varnothing\},\ i\neq\varnothing,
    \end{cases}$
\end{center}
\end{minipage}
\begin{minipage}{0.45\textwidth}
\begin{center}
    $g(i) = \begin{cases}
        \{\varnothing\},\  i = \varnothing\\
        \varnothing,\ i\neq\varnothing,
        \end{cases}$
\end{center}
\end{minipage}

\medskip

\noindent
Then, $f, g \in \mathbf{On}^\mathfrak{M}$, but $\mathfrak{M} \not\models f \in g$ and $\mathfrak{M} \not\models g \in f$. Nevertheless,

\begin{center}
$\mathfrak{M} \models f \in g \vee g \in f \vee f = g$,
\end{center}

\noindent
since for any $i \in M$, either $f(i) \in^\mathcal{M} g(i)$, $g(i) \in^\mathcal{M} f(i)$, or $f(i) = g(i)$.

Nevertheless, considering that  $\mathfrak{M}$ is a model of the arbitrary sets of $\mathcal{M}$, an arbitrary ordinal of $\mathcal{M}$ should be an arbitrary object whose values are ordinals in $\mathcal{M}$. In other words, the ordinals of a model of arbitrary sets are precisely the arbitrary ordinals of that model.

Similarly, cardinality is defined using the internal ordinal order. Let $\operatorname{Bij}(F,x,y)$ express that $F$ is a bijection from $x$ onto $y$, and put

\begin{center}
$\mathbf{C}_f
 :=\{g\in\mathbf{On}^{\mathfrak{M}}\mid
       \mathfrak{M}\models\exists F\,
                    \operatorname{Bij}(F,f,g)\}$.
\end{center}

\noindent
By Proposition \ref{prop:fundamental}, $g\in\mathbf{C}_f$ exactly when $f(i)$ and $g(i)$ are equinumerous in $\mathcal{M}$ for every $i\in M$. An internal bijection is therefore represented by a function whose values are ground-model bijections, rather than by an external bijection between collections of globally true members.

\begin{proposition}
For every $f\in\mathfrak{M}$, the function

\begin{center}
$g(i)=\operatorname{card}^{\mathcal{M}}(f(i))
\qquad(i\in M)$
\end{center}

\noindent
belongs to $\mathbf{C}_f$ and satisfies $\mathfrak{M}\models g\leq h$ for every $h\in\mathbf{C}_f$. It is the unique member of $\mathbf{C}_f$ with this property.
\end{proposition}

\begin{proof}
For every $i\in M$, $g(i)$ is the least ordinal equinumerous with $f(i)$ in $\mathcal{M}$. Proposition~\ref{prop:fundamental} therefore gives $g\in\mathbf{C}_f$. If $h\in\mathbf{C}_f$, then $\mathcal{M}\models g(i)\leq h(i)$ for every $i\in M$, so the same proposition yields $\mathfrak{M}\models g\leq h$. If $g'$ has the same property, comparison in both directions gives $g(i)=g'(i)$ for every $i\in M$, and hence $g=g'$.
\end{proof}

Consequently, the assignment

\begin{center}
$\operatorname{card}^{\mathfrak{M}}(f)(i)
   :=\operatorname{card}^{\mathcal{M}}(f(i))$
\end{center}

\noindent
is total, and

\begin{center}
$\begin{aligned}
\mathbf{Card}^{\mathfrak{M}}
 &:=\operatorname{ran}(\operatorname{card}^{\mathfrak{M}})\\
 &=\{f\in M^M\mid\forall i\in M\,
                       f(i)\in\mathbf{Card}^{\mathcal{M}}\}.
\end{aligned}$
\end{center}

\noindent
Thus the cardinals of the model are precisely the arbitrary ground-model cardinals. The minimality just proved uses the internal non-strict ordinal comparison. It does not say that every other member of $\mathbf{C}_f$ is either globally equal to $g$ or has $g$ as a globally true member: different coordinates may witness equality and strict inequality.

For arbitrary cardinals $\kappa,\eta$, comparison is likewise computed internally:

\begin{center}
$\begin{aligned}
\mathfrak{M}\models\kappa\leq\eta
 &\Longleftrightarrow
     \forall i\in M\,
       (\mathcal{M}\models\kappa(i)\leq\eta(i))\\
 &\Longleftrightarrow
     \mathfrak{M}\models\exists F\,
                (F\text{ is an injection from }\kappa\text{ into }\eta).
\end{aligned}$
\end{center}

The following result shows that an arbitrary cardinal can be globally bounded by two constant cardinals without being globally equal to either.

\begin{theorem}\label{theorem gch fails}
Let $\kappa, \eta \in \mathbf{Card}^\mathcal{M}$ be such that $\mathcal{M} \models \kappa < \eta$. Then, there is $\tau \in \mathbf{Card}^\mathfrak{M}$ such that $\mathfrak{M} \models \delta(\kappa) \leq \tau \leq \delta(\eta)$, but $\mathfrak{M} \not\models \tau = \delta(\kappa)$ and $\mathfrak{M} \not\models \tau = \delta(\eta)$.
\end{theorem}




\begin{proof}
Fix $i_0 \in M$ and define $\tau \in M^M$ by

\begin{center}
$\tau(i)\coloneqq
\begin{cases}
\kappa,&i=i_0,\\
\eta,&i\neq i_0.
\end{cases}$
\end{center}

\noindent
For every $i \in M$, $\tau(i)$ is a cardinal of $\mathcal{M}$ and
$\mathcal{M}\models\kappa\leq\tau(i)\leq\eta$. By Proposition~\ref{prop:fundamental}, we therefore have $\tau\in\mathbf{Card}^{\mathfrak{M}}$ and

\begin{center}
$\mathfrak{M}\models
\delta(\kappa)\leq\tau\leq\delta(\eta)$
\end{center}

\noindent
Moreover, since $\kappa\neq\eta$,

\begin{center}
$\llbracket\tau=\delta(\kappa)\rrbracket=\{i_0\},
\qquad
\llbracket\tau=\delta(\eta)\rrbracket=M\setminus\{i_0\}.$
\end{center}
Since $M$ has at least two elements, neither truth value is $M$. Hence $\mathfrak{M}\not\models\tau=\delta(\kappa)$ and $\mathfrak{M}\not\models\tau=\delta(\eta)$, as required.
\end{proof}


For a concrete instance of Theorem~\ref{theorem gch fails}, consider $\delta(\omega^{\mathcal{M}})$ and $\delta(\omega_1^{\mathcal{M}})$. Choose a nonempty proper subset $A \subsetneq M$ and define
\[
\tau(i)=
\begin{cases}
\omega^{\mathcal{M}},&i\in A,\\
\omega_1^{\mathcal{M}},&i\notin A.
\end{cases}
\]
Then $\tau\in\mathbf{Card}^{\mathfrak{M}}$ and $\mathfrak{M}\models
\delta(\omega^{\mathcal{M}})
\leq\tau\leq
\delta(\omega_1^{\mathcal{M}}).$ Moreover,
\[
\llbracket\tau=\delta(\omega^{\mathcal{M}})\rrbracket=A,
\mbox{ and }
\llbracket\tau=\delta(\omega_1^{\mathcal{M}})\rrbracket
=M\setminus A.
\]
Neither equality therefore holds globally, although
\[
\mathfrak{M}\models
\tau=\delta(\omega^{\mathcal{M}})
\vee
\tau=\delta(\omega_1^{\mathcal{M}}).
\]

This does not conflict with the fact that $\delta(\omega_1^{\mathcal{M}})$ is the first uncountable cardinal internally. Indeed,
\[
\llbracket
\delta(\omega^{\mathcal{M}})<\tau
\wedge
\tau<\delta(\omega_1^{\mathcal{M}})
\rrbracket
=(M\setminus A)\cap A=\varnothing.
\]
The apparent intermediate is obtained by varying between the two endpoint values across states. Internally, there is still no cardinal strictly between $\omega$ and $\omega_1$. 
More generally, Proposition~\ref{prop:fundamental} gives
\[
\mathfrak{M}\models\mathrm{GCH}
\quad\Longleftrightarrow\quad
\mathcal{M}\models\mathrm{GCH}.
\]
Thus GCH holds internally exactly when it holds in the ground model. The existence of globally distinct intermediates does not change its internal truth value.

\subsection{Arbitrary sets, in general}

The preceding two subsections have shown that the pre-theoretic notions of arbitrary ordinals, cardinals, and members of a given level of the von Neumann hierarchy match with their respective model-theoretic counterparts: ordinals of the model of arbitrary sets are precisely the arbitrary ordinals (that is, the arbitrary sets whose values are all ordinals), the elements of the $\alpha$th level of the cumulative hierarchy of the model of arbitrary sets are exactly the arbitrary sets of the $\alpha$th level, and so on.

In fact, this phenomenon is easily extended to any definable set of the ground model $\mathcal{M}$. More precisely, let a given set-theoretic concept be defined by the formula $\varphi(x)$ (such as being a limit regular cardinal, being an ultrafilter, being ordered, etc.). Then, the arbitrary $\varphi$'s (with respect to $\mathcal{M}$) are precisely the $\varphi$'s of $\mathfrak{M}$. In other words, the arbitrary sets whose values all satisfy $\varphi$ in $\mathcal{M}$ satisfy $\varphi$ in $\mathfrak{M}$. More specifically, that means the PGA may be correctly applied in $\mathfrak{M}$.

\begin{proposition}[PGA]\label{prop pga}
For every $\mathcal{L}_\in$-formula $\varphi(x_1, \ldots, x_n)$ and every $f_1, \ldots, f_n \in \mathfrak{M}$,

\begin{center}
$\mathfrak{M} \models \varphi(f_1, \ldots, f_n) \Longleftrightarrow \forall i \in M (\mathcal{M} \models \varphi(f_1(i),\ldots, f_n(i)) \big{)}$.
\end{center}
\end{proposition}

\begin{proof}
Immediate from Proposition \ref{prop:fundamental}.
\end{proof}

Proposition \ref{prop pga} is exactly the broadest reading of the PGA: a formula about arbitrary sets $\varphi(f_1, \ldots, f_n)$ holds in $\mathfrak{M}$ iff it holds of the values $f_1(i),\ldots,f_n(i)$ at every state $i \in M$.

\subsection{Crisp relations}\label{subsec:crisp}

It is useful to distinguish between \emph{truth values} and \emph{global truth}. For any formula $\varphi$ we have a truth value $\llbracket \varphi \rrbracket \subseteq M$, and we declared $\mathfrak{M} \models\varphi$ iff $\llbracket \varphi \rrbracket = M$.
Accordingly, we write

\[
a =_\Box b \quad :\Longleftrightarrow \quad \mathfrak{M} \models a = b
\qquad\text{and}\qquad
a \in_\Box b \quad :\Longleftrightarrow \quad \mathfrak{M} \models a \in b.
\]

\noindent
By Proposition \ref{prop:fundamental}, these are simply the \emph{pointwise} relations:

\[
a =_\Box b\ \Longleftrightarrow\ \forall i \in M (a(i) = b(i)),
\qquad
a \in_\Box b\ \Longleftrightarrow\ \forall i \in M (a(i) \in b(i)).
\]




We use the same convention for inclusion. Then, $x\subseteq_\Box y$ means
$\mathfrak{M} \models x\subseteq y$. For arbitrary ordinals, and hence for arbitrary cardinals, define
\[
\kappa\leq_\Box\lambda
  \quad:\Longleftrightarrow\quad
       \mathfrak{M}\models\kappa\leq\lambda.
\]
We reserve $<_\Box$ for the strict part of this external partial order:
\[
\kappa<_\Box\lambda
  \quad:\Longleftrightarrow\quad
  \kappa\leq_\Box\lambda\ \text{and}\
                         \neg(\lambda\leq_\Box\kappa).
\]
By Proposition~\ref{prop:fundamental},
\[
\kappa<_\Box\lambda
\quad\Longleftrightarrow\quad
\bigl(\forall i\in M\;\kappa(i)\leq\lambda(i)\bigr)
\ \text{and}\
\bigl(\exists i\in M\;\kappa(i)<\lambda(i)\bigr).
\]
This should be distinguished from the stronger assertion
\[
\mathfrak{M}\models\kappa<\lambda
\quad\Longleftrightarrow\quad
\forall i\in M\;\kappa(i)<\lambda(i).
\]
For cardinals, the pointwise comparisons are those computed in the ground model, equivalently by ground-model injections.



In an ordinary two-valued model of ZFC, the fact that ordinals are linearly ordered by $\in$ implies that for any two ordinals $\alpha, \beta$, exactly one of $\alpha \in \beta$, $\alpha = \beta$, or $\beta \in \alpha$ holds. In our $\mathcal{P}(M)$-valued semantics, ZFC still proves (and hence $\mathfrak{M}$ still satisfies) the \emph{trichotomy disjunction}

\[
\forall \alpha \forall \beta \big{(} \mathrm{Ord}(\alpha) \wedge \mathrm{Ord}(\beta) \to (\alpha \in \beta\ \vee\ \alpha = \beta\ \vee\ \beta \in \alpha) \big{)},
\]

\noindent
but this no longer forces one fixed disjunct to hold \emph{globally} (i.e.\ with truth value $M$). Different indices may witness different disjuncts.

\begin{proposition}\label{prop:incomparable-ordinals}
There exist $f, g\in \mathbf{On}^{\mathfrak{M}}$ such that neither $f \in_\Box g$ nor $g \in_\Box f$ nor $f =_\Box g$, while $\mathfrak{M} \models (f \in g \vee g\in f \vee f = g)$.
\end{proposition}

\begin{proof}
Fix an index $i_0 \in M$. Define $f, g \in \mathfrak{M}$ by

\[
f(i) = \begin{cases}
\varnothing & i = i_0,\\
\{\varnothing\} & i \neq i_0
\end{cases}
\qquad
g(i)= \begin{cases}
\{\varnothing\} & i = i_0,\\
\varnothing & i \neq i_0
\end{cases}
\]

\noindent
Then for every $i$, both $f(i)$ and $g(i)$ are ordinals of $\mathcal{M}$, so by our previous observations, $f, g \in \mathbf{On}^{\mathfrak{M}}$.

At $i = i_0$ we have $f(i_0) = \varnothing \in \{\varnothing\} = g(i_0)$, while at $i \neq i_0$ we have $g(i) = \varnothing \in \{\varnothing\} = f(i)$. Hence $\llbracket f \in g \rrbracket$ is the singleton $\{i_0\}$ and $\llbracket g \in f \rrbracket = M \setminus \{i_0\}$, so neither has truth value $M$, i.e.\ neither $f \in_\Box g$, nor $g \in_\Box f$, nor $f =_\Box g$. But

\[
\llbracket f \in g\ \vee\ g \in f\ \vee\ f = g \rrbracket
=
\llbracket f \in g \rrbracket \cup \llbracket g \in f \rrbracket \cup \llbracket f = g \rrbracket
=
M,
\]

\noindent
so $\mathfrak{M} \models (f \in g \vee g \in f \vee f = g)$. Finally, $f \neq_\Box g$ since $f(i_0) \neq g(i_0)$.
\end{proof}

\subsection{Boolean-valued models of set theory}

Given our Boolean-valued semantics for the direct-power model of arbitrary sets, a comparison with Boolean-valued models of set theory is in order. To this end, let us briefly recall their construction \citep{Bell2011}.

Given a complete Boolean algebra $\mathbb{B}= \langle \mathbf{A}; \wedge, \vee, \Rightarrow, ^* , 1,0 \rangle$, the corresponding universe of $\mathbb{B}$-valued sets is defined as follows:

\medskip

\[
\begin{aligned}
V^{\mathbb{B}}_0
  &= \varnothing,\\
V^{\mathbb{B}}_{\alpha+1}
  &= \{\, f : X \to \mathbb{B} \mid X \subseteq V^{\mathbb{B}}_\alpha \,\},\\
V^{\mathbb{B}}_\lambda
  &= \bigcup_{\alpha < \lambda} V^{\mathbb{B}}_\alpha
     \quad \text{for limit } \lambda,\\
V^{\mathbb{B}}
  &= \bigcup_{\alpha \in \mathbf{On}} V^{\mathbb{B}}_\alpha.
\end{aligned}
\]

\medskip

We can extend the language of set theory $\mathcal{L}_{\in}$ with constants for each  $\mathbb{B}$-valued set. The resulting language is denoted by $\mathcal{L}_{\mathbb{B}}$.  Then, for any $f,g\in V^{\mathbb{B}}$, the atomic formulas of $\mathcal{L}_{\mathbb{B}}$ are evaluated in $V^{\mathbb{B}}$ as follows:
\[
\begin{aligned}
\llbracket f \in g \rrbracket
  &= \bigvee_{x \in \operatorname{dom}(g)}
     \bigl(g(x) \wedge \llbracket f = x \rrbracket\bigr),\\[0.5em]
\llbracket f = g \rrbracket
  &= \bigwedge_{x \in \operatorname{dom}(f)}
     \bigl(f(x) \Rightarrow \llbracket x \in g \rrbracket\bigr)
     \wedge
     \bigwedge_{x \in \operatorname{dom}(g)}
     \bigl(g(x) \Rightarrow \llbracket x \in f \rrbracket\bigr).
\end{aligned}
\]
The remaining formulas $\varphi \in \mathcal{L}_{\mathbb{B}}$ are evaluated as follows:

\[
\begin{aligned}
\llbracket \neg \varphi \rrbracket
  &=  (\llbracket \varphi \rrbracket)^*,\\
\llbracket \varphi \wedge \psi \rrbracket
  &= \llbracket \varphi \rrbracket \wedge \llbracket \psi \rrbracket,\\
\llbracket \varphi \vee \psi \rrbracket
  &= \llbracket \varphi \rrbracket \vee \llbracket \psi \rrbracket,\\
\llbracket \varphi \to \psi \rrbracket
  &= \llbracket \varphi \rrbracket \Rightarrow \llbracket \psi \rrbracket,\\
\llbracket \varphi \leftrightarrow \psi \rrbracket
  &= \llbracket \varphi \to \psi \rrbracket
     \wedge
     \llbracket \psi \to \varphi \rrbracket,\\
\llbracket \exists x\, \varphi(x) \rrbracket
  &= \bigvee_{f \in V^{\mathbb{B}}}
     \llbracket \varphi(f) \rrbracket,\\
\llbracket \forall x\, \varphi(x) \rrbracket
  &= \bigwedge_{f \in V^{\mathbb{B}}}
     \llbracket \varphi(f) \rrbracket.
\end{aligned}
\]

\noindent
A formula $\varphi$ of the extended language is satisfied in $V^{\mathbb{B}}$ when the value assigned to $\varphi$ is the top element of $\mathbb{B}$, that is, when $\llbracket \varphi \rrbracket = 1$.

\citep{Horsten2019} argues that Boolean-valued sets, for a complete Boolean algebra $\mathbb{B}$, can be understood as arbitrary objects. Each $u \in V^\mathbb{B}$ determines a function $u^* : \mathbb{B} \to V^\mathbb{B}$ associating to each state $b \in \mathbb{B}$ the unique $V^\mathbb{B}$-set that $u$ becomes in that state. The state space is thus indexed by $\mathbb{B}$, and the valuation is given by $\mathrm{v}(u, b) = u^*(b)$. The proposal is attractive as an argument for understanding Boolean-valued sets as arbitrary objects \emph{simpliciter}. However, it faces the following structural obstacle, if one wants to argue that Boolean-valued sets can represent arbitrary sets in general---that is, any arbitrary set of a given model. For a Boolean-valued set $u \in V^\mathbb{B}$ to be a general arbitrary set, its scope must be all of $V^\mathbb{B}$; but the scope of $u$ is determined by $\mathrm{dom}(u)$, which is a subset of $V^\mathbb{B}$ defined by transfinite recursion on the rank of $u$. For some $u$ to represent a general arbitrary set, every element of $V^\mathbb{B}$ would need to appear in $\mathrm{dom}(u)$, requiring $\mathbb{B}$ to be at least as large as $V^\mathbb{B}$ itself.

The first difference we may notice between $V^{\mathbb{B}}$ and $\mathfrak{M}$ is the following. In the former, the elements of the model are partial functions whose domains are contained in $V^{\mathbb{B}}$ and whose values are elements of the Boolean algebra $\mathbb{B}$ used in the semantics. In the latter, by contrast, the elements are total functions with domain $M$ and codomain $M$, rather than $\mathcal{P}(M)$, the Boolean algebra of truth values. In other words, elements of $V^{\mathbb{B}}$ have recursively defined domains, whereas elements of $\mathfrak{M}$ have a fixed domain and a fixed codomain. Thus, the latter may be simpler to work with, since it avoids this recursive complexity.

Furthermore, $\mathfrak{M}$ contains the identity function $\mathrm{id}_{\mathcal{M}}$, which may take any element of $\mathcal{M}$ as its value at some state and therefore represents a general arbitrary set. There is no natural analogue in $V^{\mathbb{B}}$ of such an arbitrary set.

We now show that, despite these differences, the bounded arbitrary sets in $\mathfrak{M}$ embed faithfully into $V^{\mathcal{P}(M)}$. Let $V^\mathfrak{M} = \bigcup\limits_{\alpha \in \mathbf{On}} V^\mathfrak{M}_\alpha \neq \mathfrak{M}$.

\begin{theorem}\label{theorem embedding boolean model}
There is a Boolean-valued embedding $\mathfrak{M}\upharpoonright_{V^{\mathfrak{M}}} \hookrightarrow V^{\mathcal{P}(M)}$.
\end{theorem}

\begin{proof}
Let $\mathbb{B}=\mathcal{P}(M)$. For $f\in V^{\mathfrak{M}}$, let $\rho(f)$ be the least ordinal $\alpha$ such that $f\in V^{\mathfrak{M}}_{\alpha+1}$. Thus, for every $i\in M$, $f(i)\subseteq \mathcal{M}_{\rho(f)}.$ We define $e:V^{\mathfrak{M}}\to V^{\mathbb{B}}$ by transfinite recursion on $\rho(f)$. If $f_{\varnothing}$ is the constant function with value $\varnothing$, put $e(f_{\varnothing})=\varnothing.$ More generally, for $f\in V^{\mathfrak{M}}$, define $e(f)$ by

\[
\operatorname{dom}(e(f))
 =
\{\,e(h)\mid h\in V^{\mathfrak{M}}_{\rho(f)}\,\},
\]

\noindent
and, for $h\in V^{\mathfrak{M}}_{\rho(f)}$, put

\[
e(f)(e(h))
 =
\llbracket h\in f\rrbracket_{\mathfrak{M}}.
\]

\noindent
Equivalently,
\[
e(f)
 =
\{\,\langle e(h),\llbracket h\in f\rrbracket_{\mathfrak{M}}\rangle
   \mid h\in V^{\mathfrak{M}}_{\rho(f)}\,\}.
\]

\noindent
The induction below shows that $e$ is an embedding.

We shall use the following elementary observation. For every ordinal $\alpha$, every $i\in M$, and every $x\in \mathcal{M}_{\alpha}$, there is some $h\in V^{\mathfrak{M}}_{\alpha}$ such that $h(i)=x.$ Indeed, if  $\alpha=0$, the claim is vacuous. If $\alpha=\beta+1$, then $x\in \mathcal{M}_{\beta+1}$, so the constant function with value $x$ belongs to $V^{\mathfrak{M}}_{\beta+1}$. If $\alpha$ is a limit ordinal, then $x\in \mathcal{M}_{\beta}$ for some $\beta<\alpha$, and the claim follows from the induction hypothesis.

We now prove, by induction on $\alpha$, that for all

$f,g\in V^{\mathfrak{M}}_{\alpha}$,
\[
\llbracket f=g\rrbracket_{\mathfrak{M}}
 =
\llbracket e(f)=e(g)\rrbracket_{\mathbb{B}}
\]

\noindent
and

\[
\llbracket f\in g\rrbracket_{\mathfrak{M}}
 =
\llbracket e(f)\in e(g)\rrbracket_{\mathbb{B}}.
\]

\noindent
The case $\alpha=0$ is vacuous. The limit case follows immediately from the induction hypothesis. So suppose the result holds at stage $\alpha$, and let $f,g\in V^{\mathfrak{M}}_{\alpha+1}$. Since the base case is covered, assume $\rho(f), \rho(g) > 0$, and let $\mu = \rho(f)-1$ and $\nu = \rho(g)-1$. Thus $\mu, \nu\leq \alpha$.

First we prove the equality clause. By the definition of equality in $V^{\mathbb B}$, and by the definition of
$e$, we have

\[
\begin{aligned}
\llbracket e(f)=e(g)\rrbracket_{\mathbb{B}}
 &=
 \bigcap_{y\in V^{\mathfrak{M}}_{\mu}}
 \Bigl(
   \llbracket y\in f\rrbracket_{\mathfrak{M}}
   \Rightarrow
   \llbracket e(y)\in e(g)\rrbracket_{\mathbb{B}}
 \Bigr) \\
 &\quad\cap
 \bigcap_{y\in V^{\mathfrak{M}}_{\nu}}
 \Bigl(
   \llbracket y\in g\rrbracket_{\mathfrak{M}}
   \Rightarrow
   \llbracket e(y)\in e(f)\rrbracket_{\mathbb{B}}
 \Bigr).
\end{aligned}
\]

\noindent
Since $y\in V^{\mathfrak{M}}_{\mu}\subseteq V^{\mathfrak{M}}_{\alpha}$, the induction hypothesis gives

\[
\llbracket e(y)\in e(g)\rrbracket_{\mathbb{B}}
 =
\llbracket y\in g\rrbracket_{\mathfrak{M}},
\]

\noindent
and similarly,

\[
\llbracket e(y)\in e(f)\rrbracket_{\mathbb{B}}
 =
\llbracket y\in f\rrbracket_{\mathfrak{M}}.
\]

\noindent
Hence

\[
\begin{aligned}
\llbracket e(f)=e(g)\rrbracket_{\mathbb{B}}
 &=
 \bigcap_{y\in V^{\mathfrak{M}}_{\mu}}
 \Bigl(
   \llbracket y\in f\rrbracket_{\mathfrak{M}}
   \Rightarrow
   \llbracket y\in g\rrbracket_{\mathfrak{M}}
 \Bigr) \\
 &\quad\cap
 \bigcap_{y\in V^{\mathfrak{M}}_{\nu}}
 \Bigl(
   \llbracket y\in g\rrbracket_{\mathfrak{M}}
   \Rightarrow
   \llbracket y\in f\rrbracket_{\mathfrak{M}}
 \Bigr).
\end{aligned}
\]

\noindent
Now let $i\in M$. Then $i$ belongs to the right-hand side if and only if
\[
\forall y\in V^{\mathfrak{M}}_{\mu}\,
\bigl(y(i)\in f(i)\Rightarrow y(i)\in g(i)\bigr)
\]

\noindent
and

\[
\forall y\in V^{\mathfrak{M}}_{\nu}\,
\bigl(y(i)\in g(i)\Rightarrow y(i)\in f(i)\bigr).
\]

\noindent
Since $f(i)\subseteq \mathcal{M}_{\mu}$ and $g(i)\subseteq \mathcal{M}_{\nu}$, the observation above implies that this is equivalent to $f(i)\subseteq g(i) \mbox{ and } g(i)\subseteq f(i)$. Thus it is equivalent to $f(i)=g(i)$. Therefore

\[
\llbracket e(f)=e(g)\rrbracket_{\mathbb{B}}
 =
\{\,i\in M\mid f(i)=g(i)\,\}
 =
\llbracket f=g\rrbracket_{\mathfrak{M}}.
\]

We now prove the membership clause. By the definition of membership in

$V^{\mathbb{B}}$,
\[
\begin{aligned}
\llbracket e(f)\in e(g)\rrbracket_{\mathbb{B}}
 &=
 \bigcup_{y\in V^{\mathfrak{M}}_{\nu}}
 \Bigl(
   e(g)(e(y))
   \cap
   \llbracket e(f)=e(y)\rrbracket_{\mathbb{B}}
 \Bigr) \\
 &=
 \bigcup_{y\in V^{\mathfrak{M}}_{\nu}}
 \Bigl(
   \llbracket y\in g\rrbracket_{\mathfrak{M}}
   \cap
   \llbracket f=y\rrbracket_{\mathfrak{M}}
 \Bigr).
\end{aligned}
\]

\noindent
Hence

\[
\begin{aligned}
\llbracket e(f)\in e(g)\rrbracket_{\mathbb{B}}
 &=
 \bigcup_{y\in V^{\mathfrak{M}}_{\nu}}
 \Bigl(
   \{\,i\in M\mid y(i)\in g(i)\,\}
   \cap
   \{\,i\in M\mid f(i)=y(i)\,\}
 \Bigr) \\
 &=
 \{\,i\in M\mid f(i)\in g(i)\,\} \\
 &=
 \llbracket f\in g\rrbracket_{\mathfrak{M}}.
\end{aligned}
\]

The nontrivial inclusion in the penultimate equality obtains as follows. If $f(i)\in g(i)$, then, since $g(i)\subseteq \mathcal{M}_{\nu}$, we have $f(i)\in \mathcal{M}_{\nu}$. By the observation above, there is some $y\in V^{\mathfrak{M}}_{\nu}$ such that $y(i)=f(i)$. Hence $i\in \llbracket y\in g\rrbracket_{\mathfrak{M}} \cap \llbracket f=y\rrbracket_{\mathfrak{M}},$ as required.

Finally, $e$ is injective. Indeed, if $e(f)=e(g)$, then

\[
M
 =
\llbracket e(f)=e(g)\rrbracket_{\mathbb{B}}
 =
\llbracket f=g\rrbracket_{\mathfrak{M}},
\]

\noindent
so $f(i)=g(i)$ for every $i\in M$. Hence $f=g$.
\end{proof}

Notice, once again, that the domain of this natural embedding is restricted to bounded arbitrary sets. Therefore, the unbounded arbitrary sets---those in $\mathfrak{M} \setminus V^\mathfrak{M}$---have no image in $V^{\mathcal{P}(M)}$. Thus, Boolean-valued sets, as Horsten has considered, may not only be taken to represent arbitrary objects, but actually may be seen to correspond to bounded arbitrary sets. Indeed, Horsten’s argument, if understood as identifying Boolean-valued sets with representations of arbitrary sets, accurately frames the bounded arbitrary sets. Nevertheless, it cannot accommodate general arbitrary sets. 

Suppose now we extend the Boolean-valued model $V^{\mathcal{P}(M)}$ by allowing functions to have the whole of $V^{\mathcal{P}(M)}$ as their domains, that is, define

\begin{center}
$\mathcal{V}^{\mathcal{P}(M)}\ =\ \{ f : X \to \mathcal{P}(M) \mid X \subseteq V^{\mathcal{P}(M)}\}$.
\end{center}

\noindent
In this way, $\mathcal{V}^{\mathcal{P}(M)}$ has proper class-sized collections---from the internal point of view of $V^{\mathcal{P}(M)}$---, such as the Boolean-valued set $\{\langle \alpha, M \rangle \mid V^{\mathcal{P}(M)} \models \alpha\ \text{is an ordinal}\}$, the collection of ordinals of $V^{\mathcal{P}(M)}$. Clearly, $V^{\mathcal{P}(M)} \subseteq \mathcal{V}^{\mathcal{P}(M)}$.

\begin{theorem}
There is a Boolean-valued embedding $\mathfrak{M} \hookrightarrow \mathcal{V}^{\mathcal{P}(M)}$.
\end{theorem}

\begin{proof}
Just consider the same embedding $e$ as defined in the proof of Theorem \ref{theorem embedding boolean model}, naturally extended to the whole $\mathfrak{M}$. The same argument shows $e$ is an embedding, having $\mathfrak{M}$ as its domain.
\end{proof}

Thus, we may conclude, the unbounded arbitrary sets---those in $\mathfrak{M} \setminus V^\mathfrak{M}$---occupy a position analogous to proper classes in NBG: they exist as objects of $\mathfrak{M}$ but do not lie at any finite rank. Unlike proper classes, however, they are first-order objects of $\mathfrak{M}$ that can serve as values of quantified variables within $\mathfrak{M}$.

We can now take stock of the relationship to Horsten's argument precisely. Horsten's construction identifies Boolean-valued sets in $V^\mathbb{B}$ with arbitrary objects, using the Boolean algebra $\mathbb{B}$ as the state space. Theorem \ref{theorem embedding boolean model} shows that this identification actually frames the bounded arbitrary sets in $\mathfrak{M}$. What Horsten's identification lacks, for arbitrary sets in general, is the unbounded arbitrary sets. In this sense, our construction extends the scope of Horsten's proposal by including unbounded arbitrary sets, such as $\mathrm{id}_{\mathcal{M}}$, whose range is all of $M$.

\section{Ultrapower reductions and comparison profiles}\label{sec:three}

Section \ref{sec:two} assigns to each formula $\varphi$ a truth value $\llbracket \varphi \rrbracket \subseteq M$, that is, the set of indices at which $\varphi$ holds in the ground model $M$ (Proposition \ref{prop:fundamental}). A standard way to recover an ordinary two-valued semantics from such truth values is to choose an ultrafilter $\mathcal{U}$ on $M$ and declare a formula ``true'' exactly when its truth set is $\mathcal{U}$-large. This yields the familiar ultrapower construction. We record it mainly to highlight what changes (and what does \emph{not} change) about ordinals and cardinals when one passes from the full $\mathcal{P}(M)$-valued model to a two-valued quotient.

\subsection{Ultrafilter reduction}

\begin{definition}
A (proper) \emph{filter} on $M$ is a nonempty family $\mathcal{F} \subseteq \mathcal{P}(M)$ closed under finite intersections and supersets, and not containing $\varnothing$. An \emph{ultrafilter} $\mathcal{U}$ is a maximal proper filter; equivalently, for every $A \subseteq M$ we have $A \in \mathcal{U}$ or $M \setminus A \in \mathcal{U}$.
\end{definition}

Fix an ultrafilter $\mathcal{U}$ on $M$. Define an equivalence relation $\equiv_\mathcal{U}$ on $\mathfrak{M}$ by

\[
f \equiv_\mathcal{U} g \quad :\Longleftrightarrow \quad \llbracket f = g \rrbracket \in \mathcal{U},
\]

\noindent
that is, $\{i \in M \mid f(i) = g(i)\} \in \mathcal{U}$. Let $\mathfrak{M} / \mathcal{U} \coloneqq \mathfrak{M} / \equiv_\mathcal{U}$, and write $[f]_\mathcal{U}$ for the equivalence class of $f$. Define membership on $\mathfrak{M} / \mathcal{U}$ by

\[
[f]_\mathcal{U} \in_\mathcal{U} [g]_\mathcal{U} \quad :\Longleftrightarrow \quad \llbracket f \in g \rrbracket \in \mathcal{U}
\]

\noindent
(that is, $\{i \in M \mid f(i) \in g(i)\}\in \mathcal{U}$). Equality in $\mathfrak{M} / \mathcal{U}$ is just equality of $\equiv_\mathcal{U}$-classes.

\begin{theorem}\label{thm:los}
For every $\mathcal{L}_\in$-formula $\varphi(\vec{x})$ and $\vec{a} \in \mathfrak{M}^n$,
\[
\mathfrak{M} / \mathcal{U} \models \varphi([\vec{a}]_\mathcal{U})
\quad \Longleftrightarrow \quad
\llbracket \varphi(\vec{a}) \rrbracket\in \mathcal{U}.
\]
\end{theorem}

\begin{proof}[Proof sketch]
This is the standard \L o\'{s} induction on formulas, using that our truth values satisfy $\llbracket\neg\varphi\rrbracket=M\setminus\llbracket\varphi\rrbracket$ and $\llbracket\exists x\,\psi(x)\rrbracket=\bigcup\limits_{u\in\mathfrak{M}}\llbracket\psi(u)\rrbracket$. See, for example, \citep[Ch.~4]{ChangKeisler1990}.
\end{proof}

\begin{corollary}\label{cor:zfc-ultra}
If $\mathcal{M} \models \mathrm{ZFC}$, then $\mathfrak{M} /\mathcal{U} \models \mathrm{ZFC}$.
\end{corollary}

\begin{proof}
Since $\mathcal{M}\models\mathrm{ZFC}$, Corollary~\ref{cor:zfc-transfer} gives $\mathfrak{M}\models\mathrm{ZFC}$. Thus each axiom (or scheme instance) has truth value $M$, which belongs to $\mathcal{U}$. By Theorem~\ref{thm:los}, it follows that $\mathfrak{M}/\mathcal{U}\models\mathrm{ZFC}$.
\end{proof}

\subsection{Profiles for arbitrary ordinals}

Let $f, g \in \mathbf{On}^{\mathfrak{M}}$. Define the three truth sets

\[
S^{<}_{f, g} \coloneqq \llbracket f \in g \rrbracket, \qquad S^{=}_{f, g} \coloneqq \llbracket f = g \rrbracket, \qquad S^{>}_{f, g} \coloneqq \llbracket g \in f \rrbracket.
\]

\noindent
We call the triple $\langle S^{<}_{f,g},S^{=}_{f,g},S^{>}_{f,g} \rangle$ the \emph{comparison profile} of $\langle f, g \rangle$.

\begin{proposition}\label{prop:ordinal-profile}
For any $f, g \in \mathbf{On}^{\mathfrak{M}}$, the sets $S^{<}_{f, g}, S^{=}_{f, g}, S^{>}_{f, g}$ are pairwise disjoint and

\[
S^{<}_{f, g}\ \cup\ S^{=}_{f, g}\ \cup\ S^{>}_{f, g}\ =\ M.
\]
\end{proposition}

\begin{proof}
Fix $i \in M$. Since $f(i), g(i) \in \mathbf{On}^\mathcal{M}$ and $\mathcal{M} \models \mathrm{ZFC}$, the ordinals of $\mathcal{M}$ are linearly ordered by membership, so exactly one of $f(i) \in g(i)$, $g(i) = f(i)$, or $g(i) \in f(i)$ holds in $\mathcal{M}$. By Proposition \ref{prop:fundamental} this is exactly the statement that $i$ lies in exactly one of the three truth sets.
\end{proof}

In a two-valued model, trichotomy forces one of these relations to hold outright. Here, by contrast, a profile may have all three components nonempty. The crisp relations introduced in Subsection \ref{subsec:crisp} correspond to \emph{degenerate} profiles:

\[
f \in_\Box g\ \Longleftrightarrow\ S^{<}_{f, g} = M,
\qquad
f =_\Box g\ \Longleftrightarrow\ S^{=}_{f, g} = M,
\qquad
g \in_\Box f\ \Longleftrightarrow\ S^{>}_{f, g} = M.
\]

\noindent
Proposition \ref{prop:incomparable-ordinals} provides a concrete example of a genuinely nontrivial profile.

Passing to an ultrapower collapses profiles: by ultrafilter maximality, exactly one of the three sets is $\mathcal{U}$-large. Combining Proposition \ref{prop:ordinal-profile} with Theorem \ref{thm:los} yields:

\begin{corollary}\label{cor:ordinal-profile-ultra}
For any $f, g \in \mathbf{On}^{\mathfrak{M}}$, exactly one of $[f]_\mathcal{U} \in_\mathcal{U} [g]_\mathcal{U}$, $[f]_\mathcal{U} = [g]_\mathcal{U}$, or $[g]_\mathcal{U} \in_\mathcal{U} [f]_\mathcal{U}$ holds in $\mathfrak{M} / \mathcal{U}$. Moreover, which one holds is determined by which component of the profile lies in $\mathcal{U}$.
\end{corollary}

\subsection{Profiles for arbitrary cardinals and ``crisp'' intermediates}

For cardinals $\kappa, \lambda \in \mathbf{Card}^{\mathfrak{M}}$ define

\[
T^{<}_{\kappa, \lambda} \coloneqq \llbracket \kappa < \lambda \rrbracket,\qquad
T^{=}_{\kappa, \lambda} \coloneqq \llbracket \kappa = \lambda \rrbracket,\qquad
T^{>}_{\kappa, \lambda} \coloneqq \llbracket \lambda < \kappa \rrbracket,
\]

\noindent
and call $(T^{<}_{\kappa, \lambda}, T^{=}_{\kappa, \lambda}, T^{>}_{\kappa, \lambda})$ the \emph{cardinal comparison profile} of $(\kappa, \lambda)$.

\begin{proposition}\label{prop:cardinal-profile}
For any $\kappa, \lambda \in \mathbf{Card}^{\mathfrak{M}}$, the sets $T^{<}_{\kappa, \lambda}, T^{=}_{\kappa, \lambda}, T^{>}_{\kappa, \lambda}$ are pairwise disjoint and

\[
T^{<}_{\kappa, \lambda}\ \cup\ T^{=}_{\kappa, \lambda}\ \cup\ T^{>}_{\kappa, \lambda}\ =\ M.
\]
\end{proposition}

\begin{proof}
As in Proposition \ref{prop:ordinal-profile}, fix $i \in M$. Then $\kappa(i), \lambda(i)$ are cardinals of $\mathcal{M}$, and ZFC proves that cardinals are linearly ordered by $\leq$ (equivalently, by $<$ and $=$). Thus exactly one of $\kappa(i) < \lambda(i)$, $\kappa(i) = \lambda(i)$, or $\lambda(i) < \kappa(i)$ holds in $\mathcal{M}$. Apply Proposition \ref{prop:fundamental}.
\end{proof}

\medskip
In an ordinary two-valued model of ZFC, there is no cardinal strictly between an infinite cardinal $\kappa$ and its successor cardinal $\kappa^+$. The corresponding assertion for $\kappa$ and $2^\kappa$ requires $2^\kappa=\kappa^+$. By Proposition~\ref{prop:fundamental}, GCH holds internally in $\mathfrak{M}$ exactly when it holds in $\mathcal{M}$. Nevertheless, the crisp relation $\le_\Box$ allows an arbitrary cardinal to vary between two endpoint values without being globally equal to either.

\begin{definition}\label{def:crisp-between}
Let $\kappa, \eta, \tau \in \mathbf{Card}^{\mathfrak{M}}$. We say that $\tau$ is a \emph{crisp intermediate} between $\kappa$ and $\eta$ if

\[
\kappa \leq_\Box \tau \leq_\Box \eta \qquad\text{and}\qquad \mathfrak{M} \not\models \tau=\kappa\ \text{and}\ \mathfrak{M} \not\models \tau=\eta.
\]
\end{definition}

Theorem \ref{theorem gch fails} shows that crisp intermediates are ubiquitous. For instance, if the value of $\kappa \in \mathbf{Card}^{\mathfrak{M}}$ is infinite in every state, then $\kappa <_\Box 2^{\kappa, \bullet}$ and there are many crisp intermediates between $\kappa$ and $2^{\kappa, \bullet}$.
Notice that if $\tau(i)=\kappa(i)$ on $A\subseteq M$ and $\tau(i)=\eta(i)$ on $M\setminus A$, then $[\tau]_{\mathcal{U}}=[\kappa]_{\mathcal{U}}$ when $A\in\mathcal{U}$, and $[\tau]_{\mathcal{U}}=[\eta]_{\mathcal{U}}$ when $M\setminus A\in\mathcal{U}$. One of these alternatives always holds. 

For an arbitrary $\tau\in M^M$, collapse to a constant representative has the precise criterion
\[
[\tau]_{\mathcal{U}}=[\delta(a)]_{\mathcal{U}}
\quad\Longleftrightarrow\quad
\{i\in M\mid\tau(i)=a\}\in\mathcal{U}
\qquad(a\in M).
\]
If $\mathcal{U}$ is $\chi$-complete and the external cardinality of $\operatorname{ran}(\tau)$ is less than $\chi$, one of these fibres belongs to $\mathcal{U}$: otherwise their complements would be fewer than $\chi$ members of $\mathcal{U}$ with empty intersection. Without this completeness assumption, a mixture over infinitely many values need not collapse to a constant. For a particular $\tau$, the answer is determined by its fibres, as in the displayed equivalence. In particular, a principal ultrafilter always selects the value at its distinguished index.

\subsection{Canonical representatives and commuting with ultrapowers}

Internal equinumerosity (``there is a bijection'') is definable in $\mathcal{L}_\in$, so it is also computed pointwise in our semantics. Consequently, the pointwise cardinal assignment from Section \ref{sec:two} is a natural \emph{canonical representative} of each equinumerosity class.

Because ``$y = \mathrm{card}(x)$'' is ZFC-definable (as the least ordinal equinumerous to $x$), it also commutes with ultrapower reduction:

\begin{proposition}\label{prop:card-commutes}
For any $f \in \mathfrak{M}$,

\[
\mathfrak{M} / \mathcal{U} \models \mathrm{card}([f]_\mathcal{U}) = [\mathrm{card}^{\mathfrak{M}}(f)]_\mathcal{U}.
\]
\end{proposition}

\begin{proof}
Let $\varphi(x, y)$ be a formula defining ``$y = \mathrm{card}(x)$'' in ZFC. Then $\mathfrak{M} \models \varphi(f,\mathrm{card}^{\mathfrak{M}}(f))$, so $\llbracket \varphi(f, \mathrm{card}^{\mathfrak{M}}(f)) \rrbracket = M$. By Theorem \ref{thm:los}, $\mathfrak{M} / \mathcal{U} \models \varphi([f]_\mathcal{U}, [\mathrm{card}^{\mathfrak{M}}(f)]_\mathcal{U})$, which is exactly the desired statement.
\end{proof}

\subsection{A cut controlling crisp mixing}

We now discuss with a simple way of \emph{constraining} crisp mixing in $\mathfrak{M}$ below a chosen threshold while preserving it above. Assume for this subsection that $\mathcal{M} \models \mathrm{ZFC} + \mathrm{GCH}$.

Fix an infinite cardinal $\kappa \in \mathbf{Card}^\mathcal{M}$. Define the class of \emph{$\kappa$-cut arbitrary cardinals} by

\[
\mathbf{Card}^{\mathfrak{M}}_{> \kappa}\ \coloneqq\ \{\delta(\eta) \in \mathbf{Card}^{\mathfrak{M}} \mid \eta \in \mathbf{Card}^\mathcal{M}\} \ \cup\ \{\eta \in \mathbf{Card}^{\mathfrak{M}} \mid \forall i \in M (\eta(i) > \kappa)\}.
\]

\noindent
Thus we keep \emph{all} constant cardinals, but we only allow \emph{non-constant} arbitrary cardinals whose values are always above $\kappa$.

\begin{proposition}\label{prop:no-mix-below}
Let $\tau = \delta(\eta)$ be a constant infinite cardinal with $\mathcal{M} \models \eta \leq \kappa$. Then there is no $\xi \in \mathbf{Card}^{\mathfrak{M}}_{>\kappa}$ such that

\[
\tau <_\Box \xi <_\Box 2^{\tau, \bullet}.
\]
\end{proposition}

\begin{proof}
Since $\mathcal{M} \models \mathrm{GCH}$ and $\eta$ is infinite, we have $2^\eta = \eta^+$. Hence $2^{\tau, \bullet} = \delta(\eta^+)$ is constant.

Suppose for a contradiction that $\tau <_\Box \xi <_\Box 2^{\tau, \bullet}$ for some $\xi \in \mathbf{Card}^{\mathfrak{M}}_{> \kappa}$. If $\xi$ is constant, then by GCH in $\mathcal{M}$ there is no constant cardinal strictly between $\eta$ and $\eta^+$, a contradiction. If $\xi$ is non-constant, then by definition of $\mathbf{Card}^{\mathfrak{M}}_{> \kappa}$ we have $\xi(i) > \kappa \geq \eta$ for all $i$, hence $\xi(i) \geq \kappa^+ \geq \eta^+$ for all $i \in M$. This contradicts $\xi <_\Box \delta(\eta^+)$.
\end{proof}

\begin{proposition}\label{prop:mix-above}
Let $\sigma \in \mathbf{Card}^{\mathfrak{M}}_{> \kappa}$ satisfy $\delta(\kappa) <_\Box \sigma$ (equivalently, $\sigma(i) > \kappa$ for all $i \in M$). Then there exists $\eta \in \mathbf{Card}^{\mathfrak{M}}_{> \kappa}$ such that

\[
\sigma <_\Box \eta <_\Box 2^{\sigma, \bullet}.
\]
\end{proposition}

\begin{proof}
Pick any index $i_0 \in M$ and define $\eta \in \mathfrak{M}$ by

\[
\eta(i) \coloneqq \begin{cases}
\sigma(i) & \text{if } i = i_0,\\
2^{\sigma(i)} & \text{if } i \neq i_0
\end{cases}
\]

\noindent
Since $\sigma(i) > \kappa$ for all $i \in M$, also $\eta(i) > \kappa$ for all such $i$, so $\eta \in \mathbf{Card}^{\mathfrak{M}}_{> \kappa}$.

Pointwise, $\sigma(i) \leq \eta(i) \leq 2^{\sigma(i)}$ for all $i \in M$, so $\sigma \leq_\Box \eta \leq_\Box 2^{\sigma, \bullet}$. To see that the inequalities are strict in the order induced by $\leq_\Box$, note that $\eta \leq_\Box \sigma$ fails at every $i \neq i_0$ (since $2^{\sigma(i)} \not\leq \sigma(i)$ in $\mathcal{M}$), hence $\sigma <_\Box \eta$. Similarly, $2^{\sigma, \bullet} \leq_\Box \eta$ fails at $i_0$ (since $2^{\sigma(i_0)} \not\leq \sigma(i_0) = \eta(i_0)$), hence $\eta <_\Box 2^{\sigma, \bullet}$.
\end{proof}

The cut $\mathbf{Card}^{\mathfrak{M}}_{>\kappa}$ is not itself a set-theoretic universe, but it is a device for isolating a purely \emph{crisp} phenomenon. Internally, $\mathfrak{M}$ inherits whatever cardinal arithmetic (including GCH) holds in $\mathcal{M}$ by transfer. Crisply, coordinatewise mixing makes ``intermediate'' objects abundant, and the cut shows one way to suppress such mixing below a chosen threshold. In the next subsection we show how to implement the same idea at the level of \emph{all sets}: there is a filtered subdomain of $\mathfrak{M}$ which still satisfies ZFC and whose class of arbitrary cardinals is exactly $\mathbf{Card}^{\mathfrak{M}}_{> \kappa}$. Finally, in an ultrapower $\mathfrak{M} / \mathcal{U}$, mixing between a cardinal and its successor is suppressed automatically, since an ultrafilter selects one component of each comparison profile.

\subsection{A filtered ZFC submodel realizing the cut}\label{subsec:filtered-submodel}

The cut $\mathbf{Card}^{\mathfrak{M}}_{> \kappa}$ from the previous subsection is formulated at the level of \emph{cardinals} and is not, by itself, a set-theoretic universe. It is nevertheless natural to ask whether one can obtain a genuine set-theoretic environment---a subdomain of $\mathfrak{M}$ in which quantifiers range over \emph{all} (available) arbitrary sets---that suppresses mixing below $\kappa$ while preserving it above. In this subsection we describe a simple construction of such a filtered ZFC submodel.

Let $\mathfrak{N} \subseteq \mathfrak{M}$ and let us interpret formulas with parameters from $\mathfrak{N}$ exactly as in Definition \ref{def:truth-values}, except that in the clauses for $\exists$ and $\forall$ we range over $\mathfrak{N}$ rather than all of $\mathfrak{M}$:

\[
\llbracket \exists x \varphi(x) \rrbracket_{\mathfrak{N}} = \bigcup_{f \in \mathfrak{N}} \llbracket \varphi(f) \rrbracket_{\mathfrak{N}}, \mbox{ and } \llbracket \forall x \varphi(x) \rrbracket_{\mathfrak{N}} = \bigcap_{f \in \mathfrak{N}} \llbracket \varphi(f) \rrbracket_{\mathfrak{N}}.
\]

\noindent
Atomic formulas and Boolean connectives are interpreted as before. We write $\mathfrak{N} \models \varphi$ just in case $\llbracket \varphi \rrbracket_{\mathfrak{N}} = M$.

\begin{proposition}\label{prop:subdomain-fundamental}
Let $\mathfrak{N}$ be such that $\delta[\mathcal{M}] \subseteq \mathfrak{N} \subseteq \mathfrak{M}$. Then for every $\mathcal{L}_\in$-formula $\varphi(\vec{x})$ and every tuple $\vec{a} \in \mathfrak{N}^n$,

\[
\llbracket \varphi(\vec{a}) \rrbracket_{\mathfrak{N}} =\{i \in M \mid \mathcal{M} \models \varphi(\vec{a}(i))\,\}.
\]

\noindent
In particular, for parameters in $\mathfrak{N}$ the truth values computed in $\mathfrak{N}$ agree with those computed in $\mathfrak{M}$:

\[
\llbracket \varphi(\vec{a}) \rrbracket_{\mathfrak{N}} = \llbracket \varphi(\vec{a}) \rrbracket.
\]
\end{proposition}

\begin{proof}
The proof is by induction on the complexity of $\varphi$, exactly as in Proposition \ref{prop:fundamental}. Atomic and Boolean cases are immediate. For quantifiers, consider $\exists x \psi(x, \vec{a})$. By definition,

\[
\llbracket \exists x \psi(x, \vec{a}) \rrbracket_{\mathfrak{N}} = \bigcup_{f \in \mathfrak{N}} \llbracket \psi(f, \vec{a}) \rrbracket_{\mathfrak{N}}.
\]

\noindent
By the induction hypothesis, $i$ lies in the union iff there exists $f \in \mathfrak{N}$ such that $\mathcal{M} \models \psi(f(i), \vec{a}(i))$. Since $\delta[\mathcal{M}] \subseteq \mathfrak{N}$, for every $m \in \mathcal{M}$ the constant function $\delta(m)$ lies in $\mathfrak{N}$ and satisfies $\delta(m)(i) = m$. Hence, for each fixed $i$, the set $\{f(i) \mid f \in \mathfrak{N}\}$ contains all of $\mathcal{M}$, and therefore

\[
(\exists f \in \mathfrak{N})\ \mathcal{M} \models \psi(f(i), \vec{a}(i)) \quad \Longleftrightarrow \quad (\exists m \in \mathcal{M})\ \mathcal{M} \models \psi(m, \vec{a}(i)).
\]

\noindent
The right-hand side is exactly $\mathcal{M} \models \exists x \psi(x, \vec{a}(i))$. Thus $i$ belongs to $\llbracket \exists x \psi(x,\ \vec{a}) \rrbracket_{\mathfrak{N}}$ iff $\mathcal{M} \models \exists x \psi(x, \vec{a}(i))$, proving the existential case. The universal case is analogous, using intersections and again the fact that $\delta[\mathcal{M}] \subseteq \mathfrak{N}$ supplies all constant test values.
\end{proof}

\begin{corollary}\label{cor:subdomain-zfc}
If $\mathcal{M} \models \mathrm{ZFC}$ and $\delta[\mathcal{M}] \subseteq \mathfrak{N} \subseteq \mathfrak{M}$, then $\mathfrak{N} \models \mathrm{ZFC}$.
\end{corollary}

\begin{proof}
Let $\varphi$ be any axiom (or instance of a scheme) of ZFC. Then $\varphi$ is a sentence, so Proposition \ref{prop:subdomain-fundamental} yields

\[
\llbracket \varphi \rrbracket_{\mathfrak{N}} = \{i \in M \mid \mathcal{M} \models \varphi\} = M.
\]

\noindent
Hence $\mathfrak{N} \models \varphi$.
\end{proof}

Now, fix an infinite cardinal $\kappa \in \mathbf{Card}^\mathcal{M}$. Recall from our observations of the hierarchy of $\mathfrak{M}$ in Section \ref{sec:two} that $\mathcal{M}_\alpha$ denotes the $\alpha$th stage of the cumulative hierarchy inside $\mathcal{M}$.

\begin{definition}\label{def:kappa-filtered}
Define

\[
\mathfrak{M}^{\kappa \text{-cut}} \ \coloneqq\ \delta[\mathcal{M}]\ \cup\ \{f \in \mathfrak{M} \mid\ \forall i \in M (f(i) \not\in \mathcal{M}_{\kappa+1})\}.
\]

\noindent
Thus every constant arbitrary set is included, and a non-constant arbitrary set is included only if its value is above the rank cut $\mathcal{M}_{\kappa+1}$ in every state.
\end{definition}

\begin{proposition}\label{prop:kappa-cut-zfc}
If $\mathcal{M} \models \mathrm{ZFC}$ then $\mathfrak{M}^{\kappa \text{-cut}} \models \mathrm{ZFC}$.
\end{proposition}

\begin{proof}
We have $\delta[M] \subseteq \mathfrak{M}^{\kappa \text{-cut}} \subseteq \mathfrak{M}$, so the claim follows from Corollary \ref{cor:subdomain-zfc}.
\end{proof}

\begin{proposition}\label{prop:rigidity-below-cut}
Let $f \in \mathfrak{M}^{\kappa \text{-cut}}$. If there is $i_0 \in M$ such that $f(i_0) \in \mathcal{M}_{\kappa+1}$, then $f = \delta(a)$ for some $a \in \mathcal{M}$. Equivalently: no non-constant arbitrary set in $\mathfrak{M}^{\kappa \text{-cut}}$ can take any value in $\mathcal{M}_{\kappa+1}$ in any state.
\end{proposition}

\begin{proof}
By Definition \ref{def:kappa-filtered}, every non-constant element $f$ of $\mathfrak{M}^{\kappa \text{-cut}}$ is such that $\forall i \in M (f(i) \not\in \mathcal{M}_{\kappa+1})$. Thus if $f(i_0) \in \mathcal{M}_{\kappa+1}$ for some $i_0$, then $f$ cannot be in the second part of the union  and must therefore lie in $\delta[\mathcal{M}]$.
\end{proof}

\begin{corollary}\label{cor:no-low-mixing}
Let $f \in \mathbf{On}^{\mathfrak{M}^{\kappa \text{-cut}}}$. If $f \leq_\Box \delta(\kappa)$ (equivalently, $f(i) \leq \kappa$ for all $i$), then $f$ is constant. Similarly, if $\eta \in \mathrm{Card}^{\mathfrak{M}^{\kappa \text{-cut}}}$ and $\eta \leq_\Box \delta(\kappa)$, then $\eta$ is constant.
\end{corollary}

\begin{proof}
Assume $f \leq_\Box \delta(\kappa)$. By Proposition \ref{prop:fundamental}, this means that for every $i \in M$ we have $\mathcal{M} \models f(i) \leq \kappa$. Since each $f(i)$ is an ordinal of $\mathcal{M}$, $f(i) \leq \kappa$ implies $f(i) \in \mathcal{M}_{\kappa+1}$ (since ordinals have rank equal to themselves). Therefore, $f(i) \in \mathcal{M}_{\kappa+1}$ for all $i$, and Proposition \ref{prop:rigidity-below-cut} yields that $f$ is constant. The cardinal case is identical, since cardinals are ordinals.
\end{proof}

\begin{proposition}\label{prop:cardinals-filtered}
\[
\mathbf{Card}^{\mathfrak{M}^{\kappa \text{-cut}}} =\mathbf{Card}^{\mathfrak{M}} \cap \mathfrak{M}^{\kappa \text{-cut}} = \mathbf{Card}^{\mathfrak{M}}_{>\kappa}.
\]
\end{proposition}

\begin{proof}
We show that an arbitrary cardinal $\eta \in \mathbf{Card}^{\mathfrak{M}}$ lies in $\mathfrak{M}^{\kappa \text{-cut}}$ iff it lies in $\mathbf{Card}^{\mathfrak{M}}_{> \kappa}$.

If $\eta$ is constant, then $\eta = \delta(\xi)$ for some $\xi \in \mathbf{Card}^\mathcal{M}$, so $\eta \in \delta[\mathcal{M}] \subseteq \mathfrak{M}^{\kappa \text{-cut}}$ and also $\eta \in \mathbf{Card}^{\mathfrak M}_{> \kappa}$ by definition. If $\eta$ is non-constant and $\eta \in \mathfrak{M}^{\kappa \text{-cut}}$, then by Definition \ref{def:kappa-filtered} we have $\eta(i) \not\in \mathcal{M}_{\kappa+1}$ for all $i \in M$. Since each $\eta(i)$ is an ordinal, $\eta(i) \not\in \mathcal{M}_{\kappa+1}$ implies $\eta(i) \geq \kappa+1$, hence $\eta(i) > \kappa$, for all $i \in M$. Thus $\eta \in \mathbf{Card}^{\mathfrak{M}}_{> \kappa}$. Conversely, if $\eta \in \mathbf{Card}^{\mathfrak{M}}_{> \kappa}$ is non-constant, then $\eta(i) > \kappa$ for all $i$, so in particular $\eta(i) \not\in \mathcal{M}_{\kappa+1}$ for all $i$. Hence $\eta \in \mathfrak{M}^{\kappa \text{-cut}}$.
\end{proof}

\begin{corollary}\label{cor:localized-mixing}
Assume $\mathcal{M} \models \mathrm{ZFC} + \mathrm{GCH}$, and let $\mathfrak{M}^{\kappa \text{-cut}}$ be as in \emph{Definition \ref{def:kappa-filtered}}. Then the conclusions of \emph{Propositions \ref{prop:no-mix-below}} and \emph{\ref{prop:mix-above}} hold already inside the genuine $\mathrm{ZFC}$ environment $\mathfrak{M}^{\kappa \text{-cut}}$.
\end{corollary}

\begin{proof}
By Proposition \ref{prop:cardinals-filtered}, the class of arbitrary cardinals available in $\mathfrak{M}^{\kappa \text{-cut}}$ is exactly $\mathbf{Card}^{\mathfrak{M}}_{> \kappa}$. Thus Propositions \ref{prop:no-mix-below} and \ref{prop:mix-above}, which are formulated purely in terms of $\mathbf{Card}^{\mathfrak{M}}_{> \kappa}$ and the crisp comparison relations, apply verbatim.
\end{proof}

\subsection{Arbitrary sets and measurable cardinals}
\label{sec:arbitrary-sets-measurability}

We conclude this section discussing the connection between arbitrary sets and large cardinals.\footnote{We thank an anonymous referee for suggesting to look at models where the corresponding $\mathbf{Ord}$  is a measurable cardinal.} For this subsection, let $\kappa>\omega$ be a cardinal and put $M=V_\kappa$. 
The arbitrary sets have the same domain as before:
\[
\mathfrak{M} = M^M = \{f \mid f : M \longrightarrow M\}.
\]
In particular, every arbitrary set is represented by a total function on $M$. Consider the arbitrary hierarchy object
\[
\widehat H:M\longrightarrow M,
\qquad
\widehat H(x)=V_{\operatorname{rank}(x)}.
\]
This function is well defined: if $x\in V_\kappa$, then $\operatorname{rank}(x)<\kappa$, and $V_\alpha\in V_\kappa$ for every $\alpha<\kappa$. Its values are transitive and have no uniform rank bound below $\kappa$, since
\[
\widehat H(\alpha)=V_\alpha
\qquad(\alpha<\kappa).
\]
We ask whether a two-valued reduction can realize this varying hierarchy as an ordinary set. Towards this aim, let
\[
S=\{V_\alpha:\alpha<\kappa\}\subseteq M,
\qquad
s:\kappa\longrightarrow M,
\qquad s(\alpha)=V_\alpha.
\]
The function $s$ indexes the hierarchy states. Since $\operatorname{rank}(V_\alpha)=\alpha$, we have
\[
\widehat H(V_\alpha)=V_\alpha
\qquad(\alpha<\kappa).
\]
Thus $\widehat H$ agrees with the general arbitrary set $\operatorname{id}_M$ on $S$. The function
\[
H:=\widehat H\circ s:\kappa\longrightarrow M,
\qquad H(\alpha)=V_\alpha,
\]
is thus a useful auxiliary representation of $\widehat H$ and since it belongs to $M^\kappa$. 

As usual, for an ultrafilter $\mathcal{U}$ on $M$, write
\[
f =_\mathcal{U} g
\quad \Longleftrightarrow\quad
\{x \in M : f(x) = g(x) \} \in \mathcal{U},
\]
and
\[
f \in_\mathcal{U} g
\quad \Longleftrightarrow\quad
\{x \in M : f(x) \in g(x)\} \in \mathcal{U}.
\]
Moreover, recall that $\delta_M(a)$ is the constant function on $M$ with value $a$. The following proposition states the recovery condition directly in the $M$-indexed model.

\begin{proposition}
Let $\kappa > \omega$ be a cardinal and $M = V_\kappa$. Then $\kappa$ is measurable if and only if there is an ultrafilter $\mathcal{U}$ on $M$ such that $S \in \mathcal{U}$ and, for every $f \in M^M$,

\begin{equation}
f \in_\mathcal{U} \widehat{H}
\quad \Longleftrightarrow\quad
\exists a \in M\; f =_\mathcal{U} \delta_M(a).
\label{eq:hierarchy-recovery}
\end{equation}

\noindent
Any such $\mathcal{U}$ is nonprincipal and $\kappa$-complete. Its induced ultrafilter on $\kappa$ is normal. The ultrapower $\mathcal{M}^M / \mathcal{U}$ is well founded, and its transitive collapse $\pi_\mathcal{U}: \mathcal{M}^M / \mathcal{U} \longrightarrow \mathcal{N}$ satisfies
\[
\pi_\mathcal{U}([\widehat{H}]_\mathcal{U}) = \pi_\mathcal{U}([\operatorname{id}_M]_\mathcal{U}) = M,
\qquad
\pi_\mathcal{U}([r]_\mathcal{U}) = \kappa,
\]
where $r(x)=\operatorname{rank}(x)$. Moreover,
\[
\mathcal{M} \prec \mathcal{N}
\qquad\text{and}\qquad
M \in N.
\]
\end{proposition}

\begin{proof}
We first relate the two index sets. Given an ultrafilter $\mathcal{W}$ on $\kappa$, define an ultrafilter $\mathcal{U}$ on $M$ by
\[
B \in \mathcal{U}
\quad \Longleftrightarrow\quad
\{\alpha < \kappa: V_\alpha \in B\} \in \mathcal{W}.
\]
Then $S\in \mathcal{U}$. Conversely, if $\mathcal{U}$ is an ultrafilter on $M$ with $S\in \mathcal{U}$, define
\[
A \in \mathcal{W}
\quad \Longleftrightarrow\quad
s[A] \in \mathcal{U}
\qquad (A\subseteq\kappa).
\]
Since $s$ is a bijection from $\kappa$ onto $S$, the map $A \mapsto s[A]$ is a Boolean algebra isomorphism from $\mathcal P(\kappa)$ onto $\mathcal P(S)$. Since $S\in \mathcal{U}$, the restriction of $\mathcal{U}$ to $\mathcal P(S)$ is an ultrafilter, and its inverse image under this isomorphism is precisely $\mathcal{W}$. Thus $\mathcal{W}$ is an ultrafilter on $\kappa$. Moreover, for $A \subseteq \kappa$ and $B \subseteq M$,
\[
s^{-1}[s[A]] = A,
\qquad
s[s^{-1}[B]] = B \cap S.
\]
Since $S\in \mathcal{U}$, we have $B\in \mathcal{U}$ if and only if $B \cap S \in \mathcal{U}$. These identities show that the two constructions are inverse to one another.

Now, the restriction to the hierarchy states induces an isomorphism
\[
\Phi : \mathcal{M}^M / \mathcal{U} \longrightarrow
 \mathcal{M}^\kappa / \mathcal{W},
\qquad
\Phi([f]_\mathcal{U})=[f\circ s]_\mathcal{W}.
\]
Indeed, equality and membership correspond by the definition of $\mathcal{U}$. Every function $g: \kappa \to M$ is obtained as $f \circ s$ for a total function $f: M \to M$: set $f(V_\alpha) = g(\alpha)$ and $f(x) = \varnothing$ outside $S$. Hence $\Phi$ is surjective as well as injective. Writing $\delta_\kappa(a)$ for the constant function on $\kappa$, condition \eqref{eq:hierarchy-recovery} is therefore equivalent to

\begin{equation}
g \in_\mathcal{W} H
\quad \Longleftrightarrow\quad
\exists a \in M\;g =_\mathcal{W} \delta_\kappa(a)
\qquad (g: \kappa \to M).
\label{eq:auxiliary-hierarchy-recovery}
\end{equation}

Suppose first that $\kappa$ is measurable, and choose a normal measure $\mathcal{W}$ on $\kappa$. If $g \in_\mathcal{W} H$, then $g(\alpha) \in V_\alpha$ on a set in $\mathcal{W}$. On that set the function $\alpha \mapsto \operatorname{rank}(g(\alpha))$ is regressive. Normality makes this rank equal to a fixed $\beta < \kappa$ on a set in $\mathcal{W}$. The corresponding values of $g$ lie in $V_{\beta+1}$. Measurability implies inaccessibility, so $|V_{\beta+1}| < \kappa$. By $\kappa$-completeness, one of these values is assumed on a set in $\mathcal{W}$. Thus $g =_\mathcal{W} \delta_\kappa(a)$ for some $a \in M$. Conversely, every $a\in M$ belongs to $V_\alpha$ for all sufficiently large $\alpha < \kappa$. Hence $g =_\mathcal{W} \delta_\kappa(a)$ implies $g \in_\mathcal{W} H$. This proves \eqref{eq:auxiliary-hierarchy-recovery}, and therefore \eqref{eq:hierarchy-recovery} for the induced $\mathcal{U}$.

Now suppose \eqref{eq:auxiliary-hierarchy-recovery} holds for an ultrafilter $\mathcal{W}$ on $\kappa$. Apply it to $\delta_\kappa(\beta)$ for each $\beta<\kappa$. Since the right-hand side holds, the tail
\[
\{\alpha < \kappa \mid \beta < \alpha\}
\]
belongs to $\mathcal{W}$. In particular, $\mathcal{W}$ is nonprincipal.

For every nonzero ordinal $\lambda < \kappa$ and every function $g:\kappa \to \lambda$, we have $g \in_\mathcal{W} H$. Indeed, whenever $\alpha > \lambda$, the value $g(\alpha)$ is an ordinal below $\lambda$, and hence
\[
\operatorname{rank}(g(\alpha))
= g(\alpha) < \lambda < \alpha.
\]
Thus $g(\alpha) \in V_\alpha = H(\alpha)$. Since the tail $\{\alpha < \kappa \mid \alpha > \lambda\}$ belongs to $\mathcal{W}$, upward closure gives $g \in_\mathcal{W} H$. The recovery condition therefore yields some $a \in M$ such that
\[
\{\alpha < \kappa \mid g(\alpha) = a\}\in \mathcal{W}.
\]
This set is nonempty, so $a < \lambda$. Consequently, every function from $\kappa$ to an ordinal below $\kappa$ has a fibre belonging to $\mathcal{W}$.

This is enough to show $\kappa$-completeness. Indeed, suppose, towards a contradiction, that $0 < \lambda < \kappa$ and $\langle A_\xi \mid \xi < \lambda \rangle$ is a family of members of $\mathcal{W}$ whose intersection does not belong to $\mathcal{W}$. Since $\mathcal{W}$ is an ultrafilter,
\[
D = \kappa \setminus \bigcap_{\xi < \lambda} A_\xi \in \mathcal{W}.
\]
Every $\alpha \in D$ fails to belong to at least one $A_\xi$. Define a total function $g: \kappa \to \lambda$ by
\[
g(\alpha)=
\begin{cases}
\min \{\xi < \lambda \mid \alpha \notin A_\xi\}, & \alpha \in D,\\
0, & \alpha \notin D.
\end{cases}
\]
For each $\xi < \lambda$, let
\[
F_\xi = \{\alpha < \kappa \mid g(\alpha) = \xi\}.
\]
If $\alpha \in F_\xi \cap D$, then the definition of $g$ gives $\alpha \notin A_\xi$. Therefore
\[
F_\xi \cap D \cap A_\xi=\varnothing.
\]
If $F_\xi$ belonged to $\mathcal{W}$, then $D, A_\xi \in \mathcal{W}$ and closure under finite intersections would imply $\varnothing \in \mathcal{W}$, a contradiction. Thus no fibre of $g$ belongs to $\mathcal{W}$. This contradicts the conclusion of the preceding paragraph. Hence every intersection of fewer than $\kappa$ members of $\mathcal{W}$ belongs to $\mathcal{W}$, proving that $\mathcal{W}$ is $\kappa$-complete.

Finally, if $g: \kappa \to \kappa$ is regressive on a set in $\mathcal{W}$, then $g(\alpha) \in V_\alpha$ on that set. Condition \eqref{eq:auxiliary-hierarchy-recovery} makes $g$ constant modulo $\mathcal{W}$. Thus $\mathcal{W}$ is normal, and $\kappa$ is measurable. Completeness and non-principality transfer to $\mathcal{U}$ along $s$.

Countable completeness gives well-foundedness of the ultrapowers. Let $\pi_\mathcal{W}: \mathcal{M}^\kappa / \mathcal{W} \to \mathcal{N}$ be the transitive collapse, and put $\pi_\mathcal{U} = \pi_\mathcal{W} \circ \Phi$. We now identify their constant classes. Since $\kappa$ is measurable, every $a \in V_\kappa$ has cardinality less than $\kappa$. If $g \in_\mathcal{W} \delta_\kappa(a)$, $\kappa$-completeness makes $g$ equal modulo $\mathcal{W}$ to a constant with value in $a$. Consequently, the members of $[\delta_\kappa(a)]_\mathcal{W}$ are exactly the constant classes $[\delta_\kappa(b)]_\mathcal{W}$ with $b \in a$. Induction on rank gives
\[
\pi_\mathcal{W}([\delta_\kappa(a)]_\mathcal{W})=a
\qquad (a \in M).
\]
By \L o\'{s}'s theorem, the constant embedding is elementary, and after collapse it is the inclusion $\mathcal{M} \prec \mathcal{N}$.

Condition \eqref{eq:auxiliary-hierarchy-recovery} now says that the members of $[H]_\mathcal{W}$ are exactly the constant classes associated with the elements of $M$. Hence
\[
\pi_\mathcal{W}([H]_\mathcal{W}) = M.
\]
Also, the members of $[\operatorname{id}_\kappa]_\mathcal{W}$ are exactly the constant ordinal classes below $\kappa$: every constant ordinal below $\kappa$ is a member modulo $\mathcal{W}$, and every function that is a member modulo $\mathcal{W}$ is regressive there and thus constant by normality. Therefore
\[
\pi_\mathcal{W}([\operatorname{id}_\kappa]_\mathcal{W}) = \kappa.
\]
Finally,
\[
\widehat H \circ s = \operatorname{id}_M \circ s = H,
\qquad
r \circ s = \operatorname{id}_\kappa.
\]
Transporting the two collapse identities through $\Phi$ proves the claims about $\widehat{H}$, $\operatorname{id}_M$, and $r$. Since $[H]_\mathcal{W}$ is an element of the ultrapower, its collapsed image $M$ is an element of $N$.
\end{proof}

The initial definitions require only that $\kappa > \omega$ be a cardinal. Whenever the recovery condition holds, the proposition implies that $\kappa$ is measurable and hence inaccessible, so $V_\kappa \models \mathrm{ZFC}$. 

The philosophical interest lies in the passage between two ways of being one over many. An arbitrary set brings objects together as the possible values of a single varying object; an ordinary set brings objects together as its members. The general arbitrary set $\operatorname{id}_M$ has every object of $M$ as a possible value. The measure concentrates on the hierarchy states, where $\operatorname{id}_M$ and $\widehat{H}$ agree, and the recovery condition ensures that the members of their common realization are precisely the original objects of $\mathcal{M}$. Their membership relations are preserved by the elementary inclusion into $\mathcal{N}$.

Thus the source universe becomes one object in a larger elementary structure:
\[
\mathcal{M} \prec \mathcal{N},
\qquad M \in N.
\]
Its varying rank is realized as the height of that universe, since $\pi_\mathcal{U}([r]_W)=\kappa$. Although $M$ was already a set in the metatheory, the reduction realizes inside $\mathcal{N}$ the totality that served as the domain of the source model. Measurability characterizes this specific recovery requirement, with concentration on the hierarchy states.

\section{From combinatorial to arbitrary}\label{sec:four}

This final section sketches a philosophical interpretation of the preceding constructions. Our aim is to connect the mathematics of Sections \ref{sec:one}--\ref{sec:three} to two familiar motifs in the philosophy of set theory: (i) the \emph{iterative} conception of set formation, and (ii) the \emph{combinatorial} (as opposed to purely logical) conception of collections. 

The guiding thought is that the function-based models developed
here admit a combinatorial interpretation in terms of independent variation across a fixed state space. Their connection with the iterative conception lies in the stratification of arbitrary sets whose values are uniformly bounded in the ground hierarchy.

It is helpful to keep in view the contrast between Sections \ref{sec:one} and \ref{sec:two}. Section \ref{sec:one} models arbitrary sets of a ground ZF(C) model $\mathcal{M}$ by using the direct power $\mathcal{M}^M$. This yields an attractive intensional picture, but it fails Extensionality and (more decisively) the full Separation scheme. Section \ref{sec:two} changes the semantic strategy: instead of extending membership alone, we interpret both $\in$ and $=$ via truth values in $\mathcal{P}(M)$. The result is a ZF(C) model in a non-stardard semantic environment. The present section focuses on what \emph{conception of set} this non-classical environment suggests.

A standard contrast in the philosophical literature distinguishes a \emph{logical} conception of collection from a \emph{combinatorial} conception (see, e.g., \citep{Incurvati2020,Parsons1974}). On a logical conception, a collection is determined by a defining condition and therefore sets are conceived as extensions of predicates, or as what satisfies a given formula. On a combinatorial conception, by contrast, the emphasis is on \emph{membership as primitive data}, thus, a set is determined by which objects are in it, independently of any description.

One central role of the combinatorial conception is to motivate the prominence of the power-set operation in the iterative hierarchy. If we think of a set $X$ as given, then $\mathcal P(X)$ can be understood as the totality of \emph{all possible ways of selecting members} of $X$. For finite $X$ this is the familiar $2^{|X|}$ counting argument. But the same idea is often invoked in the infinite case, where the intuition is that forming $\mathcal P(X)$ corresponds to making, for each $x\in X$, an independent ``include/exclude'' decision, and then collecting the results. A related picture appears already in Bernays' discussion of sets of integers and sequences as outcomes of indefinitely many independent determinations \citep{BenacerrafPutnam1983}.

A useful way of describing the same process is using \emph{characteristic functions}.  Indeed, subsets of $X$ are in bijection with functions $\chi:X\to 2$, where $2=\{0,1\}$. Thus
\[
\mathcal P(X)\ \cong\ 2^X.
\]
From this perspective, power-set formation is already a kind of \emph{function-space} construction: it is the space of all \emph{binary assignments} to positions indexed by $X$.

This observation will matter for our purposes because it lets us state, very cleanly, what changes when we move from classical ``combinations of members'' to our function-based ``arbitrary sets.'' The difference is not that we abandon combinatorics in favor of something else; rather, we generalize from $2$-valued assignments to \emph{many-valued} assignments, and we relocate the positions from being indexed by $X$ itself to being indexed by a \emph{state space}.

The informal core of the arbitrary-object idea is that an object may be \emph{state-dependent}: it takes a value at each state, and its scope is the collection of values it may assume. In this case, the state space is an index set $I$ and values live in a fixed background model $\mathcal{M}$.

\begin{definition}\label{def:state-scope}
Let $\Omega$ be a nonempty set (the state space), and let $X$ be a set of admissible values. An $\Omega$-indexed arbitrary object with values in $X$ is represented by a function $u:\Omega\to X$.
\end{definition}

This is a bare representation, but it captures the intended structure: the value $u(i)$ is what the arbitrary object assumes at state $i$, and the set $\{u(i) \mid i \in \Omega\}$ is its scope in the ordinary sense.

From the viewpoint of Definition \ref{def:state-scope}, the characteristic-function picture of $\mathcal P(X)$ becomes a special case, since $\mathcal P(X)\cong 2^X$ and therefore it becomes the space of $X$-indexed arbitrary objects whose values are binary. In other words, the classical combinatorial operation $X\mapsto\mathcal P(X)$ can be seen as an ``arbitrary-object'' operator of the form $X\mapsto 2^X$.

What differs then in considering arbitrary sets is that we fix the state space and vary the set of admissible values. We thereby obtain a complementary generalization of the assignment picture: instead of binary assignments to positions indexed by $X$, we consider $X$-valued assignments to positions indexed by $\Omega$.


\subsection{The arbitrary-object operator and the iterative analogy}

Fix a state space $\Omega$. Define an operator on sets of admissible values by

\[
\mathsf{A}_\Omega(X)\ \coloneqq\ \prescript{\Omega}{}{X},
\]

\noindent
the collection of all $\Omega$-indexed arbitrary objects with values in $X$. One may regard $\mathsf{A}_\Omega$ as the \emph{arbitrary-object operator} associated to $\Omega$.

It is useful to contrast two exponentiation patterns:

\begin{itemize}[leftmargin=2em]
\item In the classical power-set picture, we start from $X$ and form $\mathcal{P}(X) \cong 2^X$. Here the \emph{index set of positions varies} (it is $X$), while the codomain of values is fixed (it is $2$). 
\item In the arbitrary-object picture pursued in this paper, we start from $X$ and form $\mathsf{A}_\Omega(X) = \prescript{\Omega}{}{X}$. Here the \emph{index set of positions is fixed} (it is $\Omega$), while the codomain varies (it is $X$).
\end{itemize}

Both constructions are combinatorial in the same broad sense: they collect all possible assignments obtained by independent determinations at each position. But they implement different notions of ``position''. For $\mathcal{P}(X) \cong 2^X$, positions correspond to \emph{potential members} of $X$. For $\mathsf{A}_\Omega(X) = \prescript{\Omega}{}{X}$, positions correspond to \emph{states}.

This is exactly the shift that we are suggesting. Instead of thinking of a new set as formed by recombining members, we think of a new arbitrary set as formed by allowing a given value-domain to vary across states. Let's analyze this contrast in the finite case. If $|\Omega| = n$ and $|X| = m$, then $|\mathcal P(X)| = 2^m$ while $|\prescript{\Omega}{}{X}| = m^n$. Even when $\Omega$ and $X$ are of the same size, the two ways of ``counting combinations'' represent different generative ideas.

If the state space $\Omega$ carries an order, then $f : \Omega \to X$ can be seen as an $\Omega$-long \emph{list} of elements of $X$, repetitions allowed. Thus the move from $\mathcal{P}(X)$ to $X^\Omega$ can be described informally as a move from ``unordered combinations'' to ``ordered assignments.'' However, beside the order, the mathematical point is that what matters is that the function-space construction preserves \emph{positional information} about which values occur at which states.


This therefore represent a different  combinatorial view of sets, because not only we preserve a positional component, but this same aspect is what allows repetitions of values. In this sense this became a more inclusive, or better more permissive, combinatorial conception of sets. Moreover, its combinatorial essence is what motivate to include all possible functions from a state space $\Omega$ into is value space $X$, thus justifying the existence of elements that transcend the standard stratification of sets in stages produced by the powerset axiom.

As a consequence this proposal differs from the standard iterative account, as recently detailed in \citep{Button2025}. There, functions are introduced in stages, with both their arguments and their values drawn from previously available functions. On the other hand, here, the state space and the ground model are given in advance, and functions represent ways for an object to vary over that background. Interestingly enough, these two combinatorial conceptions are connected by strong set theoretical principles, mainly measurability, that allow to collapse a collection of arbitrary sets into a collection of standard sets that belong to an initial segment of a cumulative hierarchy of the height of a measurable cardinal.

\subsection{Functional stratification}

Sections \ref{sec:two}--\ref{sec:three} implement Definition \ref{def:state-scope} in a particularly systematic way. We consider the domain of $\mathfrak{M} = \mathcal{M}^M$ of all arbitrary sets with state space indexed by $M$ with values in $\mathcal{M}$. We then equip $\mathfrak{M}$ with the $\mathcal{P}(M)$-valued semantics of Definition \ref{def:truth-values}.




A central feature of the construction is its natural external stratification, inherited from the background hierarchy of $\mathcal{M}$. For each $\alpha\in M\cap\mathbf{On}$, the internal level $\mathfrak{M}_\alpha=\delta(\mathcal{M}_\alpha)$ is a single arbitrary set. Its globally true members form the external collection
\[
V_\alpha^{\mathfrak{M}}
=
\bigl\{
f\in\prescript{M}{}{M}
\mid
\forall i\in M\,\bigl(f(i)\in\mathcal{M}_\alpha\bigr)
\bigr\}.
\]
Thus, the parameter $\alpha$ bounds the complexity of the admissible values, while arbitrary variation across states is already permitted at each level. For example, $V_2^{\mathfrak{M}}$ consists of all functions $M\to\{\varnothing,\{\varnothing\}\}$ and therefore has external cardinality $2^{|M|}$. This is compatible with the internal assertion that $\mathfrak{M}_2=\delta(2)$ has exactly two elements: the external collection of globally true members and the internal arbitrary set are different objects. As the limit-stage example in Section \ref{subsection cumulative hierarchy} shows, these external levels are increasing but need not be continuous at limits: membership in a limit level need not supply a uniform bound below that limit.


This suggests a two-dimensional picture. The ordinary iterative hierarchy is naturally read as a single increasing scale of complexity. Here we have at least two interacting dimensions:

\begin{enumerate}[label=(\roman*),leftmargin=2.5em]
\item the complexity of \emph{values} (controlled by $\alpha$), and \item the entire  \emph{variation} at every state (built into the exponentiation).
\end{enumerate}

\noindent
The first dimension mirrors the classical iterative conception. The second dimension is the new ingredient: it models the idea that an arbitrary set is determined not only by what values are possible, but by the entire pattern of how values may be distributed across states.

This helps clarify the sense in which the construction is ``combinatorial.'' To specify an element $f \in \prescript{\Omega}{}{X}$ is to make, for each $i \in \Omega$, an independent determination of a value in $X$. In the limit, this realizes a maximal version of the assignment picture: \emph{all} statewise determinations are admitted.

The scope of this stratification must nevertheless be qualified. Writing $\theta=M\cap\mathbf{On}$ for the ordinal height of the ground model, we have
\[
\bigcup_{\alpha<\theta}V_\alpha^{\mathfrak{M}}
\subsetneq \prescript{M}{}{M}.
\]
Indeed, $\mathrm{id}_M$ belongs to the full function space, but for every $\alpha<\theta$, $\mathrm{id}_M(\alpha)=\alpha\notin\mathcal{M}_\alpha$. Thus it belongs to none of these external levels. The full domain is obtained by forming the function space externally, with the entire ground model already given. Equivalently, \(\mathrm{id}_M\) is obtained by assigning to each singleton state \(\{i\}\) the constant value \(i\). This uses values from throughout $M$ and presupposes their availability. Consequently, the bounded stratification does not supply an iteration generating the unbounded arbitrary sets.

\subsection{Three further questions}

The functional and combinatorial picture raises several natural questions. We close by flagging three that are suggested directly by the constructions above.

\medskip

\noindent
\textbf{(1) Filtered domains and degrees of arbitrariness.}
Section \ref{sec:three} stratifies arbitrariness by imposing a cut on the class of arbitrary cardinals, and Subsection \ref{subsec:filtered-submodel} shows that the same idea can be realized at the level of \emph{all} sets by passing to a filtered subdomain $\mathfrak{M}^{\kappa \text{-cut}} \subseteq \mathfrak{M}$ which still satisfies ZFC. This opens a wider family of questions: how sensitive are the resulting ``degrees of arbitrariness'' to the choice of threshold and, more importantly, to the choice of what counts as ``low''? For instance, one may replace the rank cut $\mathcal{M}_{\kappa+1}$ with a \emph{size} cut (e.g.\ requiring non-constant objects never to take values whose transitive closure has cardinality $< \kappa$ in $\mathcal{M}$), or one may demand stronger closure conditions on the filtered domain. A systematic study of such filtrations would clarify how profiles, crisp comparability, and the availability of intermediate objects depend on these structural choices.

\medskip

\noindent \textbf{(2) Interaction with extensionality and intensionality.} Section \ref{sec:one} shows that a naive extension of membership generates an intensional universe but breaks full Separation and Extensionality. Section \ref{sec:two} restores ZFC internally by moving to a truth-value semantics in which equality and membership are treated uniformly. There are two paths of philosophical interest worth pursuing here: first, to understand whether there is a canonical way to recover \textit{intensional} distinctions (of the sort that motivated the box-membership construction) \textit{within} the $\mathcal{P}(M)$-valued setting—for example, by treating profiles or other semantic invariants as part of the identity conditions for arbitrary sets; second, to develop a theory of intensional sets, along the lines of \citep{Myhill} and \citep{Goodman}, which adequately frames our intensional reasoning with collections.

\medskip

\noindent
\textbf{(3) Arbitrary objects and large cardinals.}
Section \ref{sec:arbitrary-sets-measurability} shows an interesting connection between arbitrary objects and measurable cardinals. In essence, the existence of a measurable cardinal is a set theoretical principle that allows to view an unbounded arbitrary object like $\operatorname{id}_M$ as a regular set, in a larger model. This observation then can be seen as a sort of bridge principles between large cardinals and arbitrary objects and, therefore, it opens up the possibility of investigation this relationship from both perspectives. On the one hand we could ask which other large cardinals we can express in terms of arbitrary objects, while on the other hand we could ask to which large cardinals properties correspond basic tenets of arbitrary objects.



\bigskip

We leave these questions for future work, but they suggest a broader point: arbitrary sets offer an alternative way of organizing collections, including infinite ones, while retaining important connections with ordinary set theory. These connections emerge at several levels: the combinatorial conception provides philosophical motivation, the cumulative hierarchy stratifies the complexity of their possible values, and large-cardinal principles illuminate the conditions under which unbounded arbitrary objects can be seen as regular sets from a larger perspective.

\end{document}